\documentclass[10pt, leqno]{article}

\usepackage{amsmath,amssymb,amsthm, epsfig}
\usepackage[colorlinks=true,urlcolor=blue,
citecolor=red,linkcolor=blue,linktocpage,pdfpagelabels,
bookmarksnumbered,bookmarksopen]{hyperref}
\usepackage{color}
\usepackage{ulem}
\usepackage{dsfont}
\usepackage{mathrsfs}
\usepackage{mathtools}
\usepackage{stackrel}
\usepackage{esint,enumerate}

\usepackage[left=2.5cm,right=2.5cm,top=2.5cm,bottom=2.5cm]{geometry}

\usepackage{comment}

\newcommand{\intav}[1]{\mathchoice {\mathop{\vrule width 6pt height 3 pt depth  -2.5pt
\kern -8pt \intop}\nolimits_{\kern -6pt#1}} {\mathop{\vrule width
5pt height 3  pt depth -2.6pt \kern -6pt \intop}\nolimits_{#1}}
{\mathop{\vrule width 5pt height 3 pt depth -2.6pt \kern -6pt
\intop}\nolimits_{#1}} {\mathop{\vrule width 5pt height 3 pt depth
-2.6pt \kern -6pt \intop}\nolimits_{#1}}}

\newtheorem{proposition}{Proposition}[section]
\newtheorem{lemma}{Lemma}[section]
\newtheorem{theorem}{Theorem}[section]
\newtheorem{remark}{Remark}[section]
\newtheorem{corollary}{Corollary}[section]
\newtheorem{definition}{Definition}[section]
\newtheorem{statement}{Statement}[section]

\numberwithin{equation}{section}

\title{\bf Boundary regularity of almost-minimizers for vectorial Alt-Caffarelli functionals with non-standard growth}
\author{Pedro Fellype Pontes\footnote{\noindent Zhejiang Normal University. School of Mathematical Sciences, Jinhua 321004 - People’s Republic of China. \noindent \texttt{E-mail address: fellype.pontes@gmail.com}}, \,\,\,\,\, Jo\~{a}o Vitor da Silva\footnote{\noindent Universidade Estadual de Campinas. Departamento de Matemática. Campinas, SP-Brazil 13083-859. \noindent \texttt{E-mail address: jdasilva@unicamp.br}}, \\  $\&$  \\Minbo Yang\footnote{\noindent  Corresponding author. Zhejiang Normal University.
School of Mathematical Sciences,
Jinhua 321004 - People's Republic of China.
 \noindent \texttt{E-mail address: mbyang@zjnu.edu.cn}}}
\date{January 2026}

\begin{document}

\maketitle

\begin{abstract}
 We prove that almost-minimizers, namely, functions satisfying a suitable variational inequality, of the Alt--Caffarelli-type functional
\[
\mathcal{J}_G({\bf v};\Omega) \coloneqq \int_\Omega \left( \sum_{i=1}^m G\big(|\nabla v_i(x)|\big) + \lambda \chi_{\{|{\bf v}|>0\}}(x) \right)\, dx,
\]
where ${\bf v} = (v_1,\dots,v_m)$ with $m \in \mathbb{N}$, enjoy optimal Lipschitz continuity up to the boundary. Here, $G$ is an $\mathcal{N}$-function satisfying suitable growth conditions, $\lambda>0$ is fixed, $\chi$ denotes the characteristic function of the indicated set, and $\Omega \subset \mathbb{R}^n$, $n\geq 2$, is a bounded Lipschitz domain.  Our results extend recent regularity theories for weakly coupled vectorial almost-minimizers associated with the $p$-Laplacian developed in \cite{BFS24} and \cite{DiPFFV24}, and yield new techniques applicable to a broad class of nonlinear one- and two-phase free boundary problems with non-standard growth. Notably, our results are new and striking even in the scalar case and for minimizers of the type studied by Mart\'{i}nez--Wolanski \cite{MW08} and da Silva \textit{et al.} \cite{daSSV2024}.

\end{abstract}

\vspace{0.2cm}
	

\section{Introduction}\label{Sec1}

Let $\Omega \subset \mathbb{R}^n$ be a bounded $C^{1,\alpha}$-domain, with $n \ge 2$, $m \in \mathbb{N}$, $\lambda>0$ a fixed constant, $\chi$ denotes the characteristic function of the indicated set and $G: [0,+\infty) \to [0,+\infty)$ an $\mathcal{N}$-function, we will deal with the up to the boundary regularity of almost-minimizers of the Alt-Caffarelli type functional 
\begin{equation}\label{DefFunctional}
    \mathcal{J}_G({\bf v};\Omega) \coloneqq \int_\Omega  \left(\sum_{i=1}^mG\big(|\nabla v_i(x)|\big) + \lambda \chi_{\{|{\bf v}|>0\}}(x)\right) dx, 
\end{equation}
over the class
    $$\mathcal{K} \coloneqq \Big\{ {\bf v} \in W^{1,G}(\Omega;\mathbb{R}^m) \; : \; {\bf v} = \mathbf{\Phi} \; \mbox{on} \; \partial \Omega \; \mbox{and} \; v_i \ge 0 \Big\}. $$
In this case, we consider ${\bf v} = (v_1, \dots, v_m)$, $|{\bf v}| = \sqrt{(v_1)^2 + \cdots + (v_m)^2}$, and $\mathbf{\Phi} = (\phi_1, \dots, \phi_m)$, $0 \le \phi_i \in W^{1,G}(\Omega)$. 

Furthermore, we would like to point out that we understand as (local) \textbf{$(\kappa,\beta)$-almost-minimizer}, for $\mathcal{J}_G$ in $\Omega$, with constant $\kappa\le\kappa_0$ and exponent $\beta>0$, with a prescribed boundary value $\mathbf{\Phi}$, a vectorial function ${\bf u} = (u_1, \dots, u_m)$, such that $\mathbf{u} - \mathbf{\Phi} \in W_0^{1,G}(\Omega;\mathbb{R}^n)$, satisfying
\begin{equation}\label{VarIneqAlmMin}
     \mathcal{J}_G({\bf u};B_r(x_0) \cap \Omega) \le \big(1+\kappa r^\beta\big)  \mathcal{J}_G({\bf v};B_r(x_0) \cap \Omega),
\end{equation}
for any ball $B_r(x_0)\subset \mathbb{R}^n$, and any ${\bf v} \in W^{1,G}(B_r(x_0)\cap \Omega ;\mathbb{R}^m)$ such that ${\bf u} = {\bf v}$ on $\partial B_r(x_0)$.

In a heuristic way, the energy of $\mathbf{u}$ in $B_r(x_{0})\cap\Omega$ may fail to be minimal among all
competitors $\mathbf{v} \in \mathbf{u} + W^{1,G}_{0}(B_r(x_{0})\cap\Omega)$, but it remains \textit{almost} minimal.

It is important to note that almost minimizers have attracted growing interest for several reasons. 
First, they can be interpreted as perturbations of true minimizers and thus arise naturally when noise or lower-order effects are present (cf. \cite{daSSV2024} and \cite{DiPFFV24}). 
Second, minimizers subject to additional constraints, for instance, fixed-volume conditions or solutions to the classical obstacle problem, can often be recast as almost minimizers of suitably unconstrained variational formulations \cite{Anzellotti1983}. Finally, the analysis of almost minimizers demands a distinct methodological viewpoint, which in turn yields techniques and insights that also enhance the understanding of genuine minimizers (cf. \cite{BFS24} and \cite{PSY}). We recommend to interested readers Smit Vega Garcia's survey \cite{SVGarcia2023} concerning recent advances on regularity to almost-minimizers of
Bernoulli-type functionals with variable coefficients. 

\bigskip
Before stating the main results of this paper, we introduce two essential definitions to ensure a clearer and more precise understanding of the framework and the objects involved.
\begin{definition}[{\bf $\mathcal{N}$-function}]\label{Def-N-function}
    A continuous function $G : [0,+\infty) \rightarrow [0,+\infty)$ is an $\mathcal{N}$-function if:
	\begin{itemize}
		\item[$(i)$] $G$ is convex;
		\item[$(ii)$] $G(t) = 0$ if, and only if, $t = 0$;
		\item[$(iii)$] $\displaystyle\lim_{t\rightarrow0}\frac{G(t)}{t}=0$ and $\displaystyle\lim_{t\rightarrow+\infty}\frac{G(t)}{t}= +\infty$.
	\end{itemize}
\end{definition} 

The next definition is a particular class of $\mathcal{N}$-functions, which brings to light natural conditions introduced by Lieberman (see \cite{L} for more details) in studying regularity
estimates of degenerate/singular elliptic PDEs of the type $-\mathrm{div} \left(g(|\nabla u|)\dfrac{\nabla u}{|\nabla u|}\right) = \mathcal{B}(x,u,\nabla u)$. This class is commonly referred to as the \textit{Lieberman class,} and can be defined as follows.

\begin{definition}[{\bf Non-degenerate class}]\label{defclasses} 
Let $G: [0,+\infty) \to [0,+\infty)$ be an $\mathcal{N}-$function. We say that $G\in\mathcal{G}(\delta,g_0)$ if 
\begin{equation}\label{Lieberman-Cond01}
G^{\prime}(t)=g(t) \quad \text{for a function} \quad g \in C^0([0,+\infty))\cap C^1((0,+\infty)),
\end{equation}
and for $0 < \delta \leq  g_0$ fixed parameters it holds
\begin{equation}\label{Ga}
		0<\delta\le \dfrac{g^{\prime}(s)s}{g(s)} \le g_0.
\end{equation}
\end{definition}

By introducing these definitions, we are now in a position to state the main result of this paper, which can be formulated as follows:

\begin{theorem}[{\bf Lipschitz regularity up to the boundary}]\label{thm:boundary-Lip}
    Let $\Omega$ be a $C^{1,\alpha}$-domain of $\mathbb{R}^n$, and $G$ be an $\mathcal{N}$-function such that $G \in \mathcal{G}(\delta,g_0)$, with $\delta>1$. Let $\mathbf{u}:\Omega \to \mathbb{R}^m$ be a $(\kappa,\beta)$-almost-minimizer of $\mathcal{J}_G$ in $\Omega$, with constant $\kappa \le \kappa_0$, exponent $\beta>\max\left\{4\left(1-\frac{\delta}{g_0}\right), n \left(\frac{g_0}{\delta}-1\right)+1\right\}$, the prescribed boundary value $\mathbf{\Phi}\in C^{1,\alpha}(\Omega, \mathbb{R}^m)$. Then $\mathbf{u}$ is Lipschitz continuous up to the boundary.
\end{theorem}

Observe that Theorem~\ref{thm:boundary-Lip} substantially broadens our understanding of how minimal assumptions on the boundary datum and the model data govern the optimal boundary regularity of almost-minimizers of \eqref{VarIneqAlmMin} (namely, Lipschitz estimates). In this way, it provides new insight into the modern theory of one-phase free boundary problems and their applications to Bernoulli-type models and related settings (cf. \cite{FerReal-Gruen2025} and \cite{MW08}).

It is worth noting that the purely interior version of Theorem \ref{thm:boundary-Lip} was previously established by the authors in \cite[see Theorem 1.3]{PSY}, for the general case $\delta>0$. The strategy adopted there, inspired by \cite{BFS24}, consists of first proving Hölder continuity and $C^{1,\alpha}$-regularity (away from the free boundary) for almost-minimizers, and then deriving a linear growth estimate near free boundary points (see Theorems 1.1, 1.2, and Proposition 4.4 in \cite{PSY} for details). As a consequence, the boundedness of the gradient of the almost-minimizer follows.

It might therefore seem natural to expect that suitable boundary analogues of these results also hold, leading directly to Lipschitz regularity up to the boundary. However, this is unfortunately not the case in general. The main obstruction arises in the proof of the linear growth estimate, which does not extend to general contact points, that is, to points lying in the intersection of the boundary of the domain with the closure of the free boundary. More precisely, as observed in \cite{BFS24}, if
    $$\liminf_{r\to 0} \dfrac{|B_r(x_0)\cap \Omega\cap\{|\mathbf{u}|=0\}|}{|B_r(x_0)\cap \Omega|}>0,$$
then one can still recover Lipschitz regularity up to the point $x_0$. In the general case, however, the blow-up argument leads to a nontrivial $g$-harmonic function in a half-space that vanishes at $x_0$. This does not contradict the homogeneous Harnack inequality, which plays a crucial role in proving linear growth at free boundary points (see \cite[Proposition 4.4]{PSY} for further discussion).

As a consequence, the approach used in \cite{PSY} cannot be directly applied in our current setting. Nevertheless, we are able to overcome this difficulty by adopting a different strategy, which still allows us to achieve the same conclusion. The core idea is to quantify how far almost-minimizers are from their corresponding $g$-harmonic replacements and to exploit this deviation in a precise way. This quantification allows us to recover linear growth near the free boundary. One of the main steps toward achieving this quantitative analysis is to establish boundary Hölder estimates for the almost-minimizer and, away from the free boundary, for its gradient. A fundamental step in this quantitative analysis is the derivation of boundary Hölder estimates for the almost-minimizers and, away from the free boundary, for their gradients. A crucial ingredient in the proof of these estimates is Lemma \ref{g-harmonic_control}, which plays the role of a boundary growth lemma. While such a result is classical in the $p$-Laplace setting (see, for instance, \cite[Lemma 3.4]{DGK}), one of the contributions of this work is its extension to the Orlicz framework. This extension allows us to handle general nonstandard growth conditions and provides the essential tool needed to establish the corresponding boundary Hölder estimates.

Furthermore, the new results established here, namely Lemma \ref{Linffctrl}, Proposition \ref{LGctrl}, and Corollary \ref{improveLinfctrl}, hold uniformly at every point of the free boundary, and not only at the contact points. Consequently, the method presented in this work also yields an alternative proof of the local Lipschitz continuity of almost-minimizers previously obtained in \cite{PSY} for the case $\delta>1$.

\bigskip

Now, we recall the contribution of Mart\'{i}nez-Wolanski \cite{MW08}, who investigated the optimization problem of minimizing
\[
\mathcal{J}_{\mathrm{G},\lambda}(u,\Omega)
    := \int_\Omega \bigl(\mathrm{G}(|\nabla u|)+\lambda\,\chi_{\{u>0\}}\bigr)\,dx
    \quad \Rightarrow \quad
\left\{
\begin{array}{ccl}
    \mathrm{div}\!\left(g(|\nabla u|)\dfrac{\nabla u}{|\nabla u|}\right) = 0 & \text{in} & \{u>0\}\cap\Omega,\\[0.2cm]
    u = 0,\;\; |\nabla u| = \lambda^{\ast} & \text{on} & \partial\{u>0\}\cap\Omega,\\[0.2cm]
    G^{\prime}(t)=g(t) & \text{and} & g(\lambda^{\ast})\lambda^{\ast}-G(\lambda^{\ast})=\lambda.
\end{array}
\right.
\]
They work in the class $W^{1,\mathrm{G}}(\Omega)$ with $u-\phi_0\in W^{1,\mathrm{G}}_0(\Omega)$, where $\phi_0\ge0$ is bounded and $\lambda>0$. The authors show that minimizers are locally Lipschitz continuous and satisfy the associated Bernoulli-type free boundary problem, extending the classical Alt--Caffarelli theory to the Orlicz--Sobolev setting. They further establish a version of Caffarelli’s classification: flat or merely Lipschitz free boundaries are locally $C^{1,\alpha}$ for some universal $\alpha\in(0,1)$.
Subsequently, in \cite{BM2014}, these results were extended to two-phase problems. The authors proved the existence of a universal threshold---depending only on the degenerate ellipticity and other intrinsic parameters---for the density of the negative phase along the free boundary, below which uniform Lipschitz regularity holds. Their analysis relies heavily on Karakhanyan’s approach in \cite{Karakhanyan2008}, where the author previously established local Lipschitz estimates for a two-phase $p$-Laplacian problem under a smallness condition on the density of the negative phase.  Free boundary problems have also been extensively studied by many authors, both in the one-phase and two-phase settings. 

\bigskip

As a consequence of our analysis, we derive global Lipschitz bounds for minimizers of $\mathcal{J}_G$. In this sense, our estimates may be viewed as the global analogue of those obtained by Mart\'{i}nez--Wolanski \cite{MW08} and da Silva \textit{et al} \cite{daSSV2024} for the scalar scenario.

\begin{corollary}[{\bf Boundary Lipschitz regularity for minimizers}]\label{Corollary1.1}
    Let $\Omega$ be a $C^{1,\alpha}$ domain in $\mathbb{R}^n$, and let $G$ be an $\mathcal{N}$-function such that $G \in \mathcal{G}(\delta,g_0)$ with $\delta>1$. Suppose $\mathbf{u}:\Omega \to \mathbb{R}^m$ is a minimizer of $\mathcal{J}_G$ in $\Omega$ with boundary data $\mathbf{\Phi}\in C^{1,\alpha}(\Omega,\mathbb{R}^m)$. Then $\mathbf{u}$ is Lipschitz continuous up to the boundary.
\end{corollary}

In particular, it is worth emphasizing that minimizers of $\mathcal{J}_G$, which are globally Lipschitz continuous due to Corollary \ref{Corollary1.1}, satisfy the associated Bernoulli-type free boundary problem
\[
\left\{
\begin{array}{rcl}
\mathrm{div}\!\left(g(|\nabla \mathbf{u}|)\dfrac{\nabla \mathbf{u}}{|\nabla \mathbf{u}|}\right) = 0 
& \text{in} & 
\displaystyle P_{\mathbf{u}} \coloneqq \bigcup_{i=1}^{m} \{u_i>0\} \cap \Omega,\\[0.2cm]
\mathbf{u} = 0, \quad |\nabla \mathbf{u}| = \lambda^{\sharp} 
& \text{on} & 
\partial P_{\mathbf{u}} \cap \Omega,\\[0.2cm]
G'(t) = g(t) 
& \text{and} & 
g(\lambda^{\sharp})\lambda^{\sharp} - G(\lambda^{\sharp}) = \lambda.
\end{array}
\right.
\]
This result extends the classical Alt--Caffarelli theory to the vectorial Orlicz--Sobolev setting, up to the boundary (cf. \cite{Fernandez-RealYu2023}).

\subsection{State-of-the-art and some motivational related results}

In this subsection, we present several relevant free boundary problems in which boundary regularity estimates play a central role. To begin with, it is worth emphasizing that solutions to many free boundary problems naturally arise as minimizers of Alt-Caffarelli-type functionals (see \cite{AltCaffarelli1981} for Alt-Caffarelli's seminal work)
$$
\displaystyle \mathscr{J}_{\textrm{AC}}(u) := \int_{\Omega \cap \{u>0\}} \mathscr{F}(x, u(x),\nabla u(x))dx  \rightarrow \min_{\mathscr{K}} \mathscr{J}_{\textrm{AC}}(v).
$$
Establishing regularity for such solutions is a delicate task, since these functionals are inherently non-smooth due to the presence of characteristic functions and free boundaries, which prevent the direct application of classical variational techniques.

An effective strategy to overcome this difficulty is based on penalization methods. The core idea of this approach is to approximate the original non-smooth functional by a family of smoother functionals, for instance, through singular perturbation schemes. These approximating problems admit minimizers with better analytical properties like Harnack inequality, good {\it a priori} estimates, comparison principle, among others, making it possible to derive uniform regularity estimates (independent of penalization scheme). By carefully passing to the limit and exploiting the uniform convergence of the approximating minimizers, one can then transfer the regularity results to solutions of the original free boundary problem. This methodology has proven to be a powerful and flexible tool in the analysis of boundary regularity for Alt–Caffarelli-type problems (see Caffarelli-Salsa’s classic monograph \cite{CS05} for such a subject).

In this direction, Karakhanyan \cite{Karakhanyan2006} studied the problem
\[
\begin{cases}
\Delta_p u_\varepsilon = \zeta_\varepsilon(u_\varepsilon) & \text{in } \Omega,\\[4pt]
0 \le u_\varepsilon \le 1 & \text{in } \Omega,\\[4pt]
u_\varepsilon(x) = g(x) & \text{on } \partial\Omega,
\end{cases}
\]
where $\Delta_p$ denotes the $p$-Laplace operator, $g \in C^{1,\alpha}(\partial \Omega)$, and $\zeta_{\varepsilon}(t) = \frac{1}{\varepsilon} \zeta\left(\frac{t}{\varepsilon}\right)$, where $0 \leq \zeta \in C_{0}^{\infty}([0, 1])$. He establishes
uniform bounds for $\nabla u_\varepsilon$. Consequently, up to a subsequence,
$u_\varepsilon \to u$ uniformly, and the limit $u$ satisfies the following one-phase free boundary
problem of Bernoulli-type in a weak sense:
\[
\begin{cases}
\Delta_p u = 0 & \text{in } \Omega \cap \{u>0\},\\[4pt]
|\nabla u| = \mathrm{c} & \text{in } \Omega \cap \partial \{u>0\}, \,\,\, (\mathrm{c}>0)\\[4pt]
u(x) = g(x) & \text{on } \partial\Omega.
\end{cases}
\]

\medskip

\medskip

In \cite{KarakShahg2016}, Karakhanyan and Shahgholian examined the boundary behaviour of the family of solutions $\{u_{\varepsilon}\}_{\varepsilon>0}$ to the singular perturbation problem
\begin{equation}\label{KS-Eq}
\Delta u_{\varepsilon} = \psi _{\varepsilon}(u_{\varepsilon}), 
\qquad |u_{\varepsilon}| \leq 1 \ \text{in } B_1^{+} := \{x_n > 0\} \cap \{|x| < 1\},
\end{equation}
where smooth boundary data $g$ are prescribed on the flat portion of $\partial B_1^{+}$. Here 
\[
\psi_{\varepsilon}(s) = \frac{1}{\varepsilon}\,\psi\!\left(\frac{s}{\varepsilon}\right),
\qquad 0 \leq \psi \in C_0^{\infty}(0,1), \quad \text{and} \quad  \int_0^1 \psi(t)\,dt = \mathrm{M}>0,
\]
is an approximation of the identity. The limit function, obtained as $\varepsilon \to 0$, locally solves the following free boundary problem in a very weak sense:
$$
\begin{cases}
\Delta u = 0 & \text{in } \{u>0\} \cup \{u<0\},\\[4pt]
(u_\nu^{+})^{2} - (u_\nu^{-})^{2} = 2\mathrm{M} & \text{on } \partial\{u>0\}.
\end{cases}
$$

In such a context, it is known that, under the assumption (a sort of degenerate phase transition)
\begin{equation}\label{DegPhasTran}
\nabla g(z) = 0 \quad \text{whenever} \quad  g(z)=0 \quad \text{on the flat portion of the boundary},    
\end{equation}
 the family $\{u_\varepsilon\}_{\varepsilon>0}$ is uniformly bounded in the Lipschitz norm; see \cite{Gurevich}. More precisely,
$$
\sup_{x \in B_{1/2}^{+}} |\nabla u_\varepsilon(x)| \le \mathfrak{L}_0,
$$
where $\mathfrak{L}_0>0$ is a constant independent of $\varepsilon$, valid for any solution of \eqref{KS-Eq}.

Recently, in \cite{BS2025}, the authors studied a singular perturbation problem driven by the $g$-Laplacian. In this setting, they established uniform Lipschitz estimates up to the boundary for a family of solutions $\{v_\varepsilon\}_{\varepsilon>0}$ to
\begin{equation}\label{SPP-BS}
\begin{cases}
\Delta_g v_\varepsilon = f + \psi_\varepsilon(v_\varepsilon) & \text{in } B_1^{+},\\[4pt]
v_\varepsilon \ge 0 & \text{in } B_1^{+},\\[4pt]
v_\varepsilon = g & \text{on } B_1'.
\end{cases}
\end{equation}
Here $f$ and $g$ are suitable data, and $\psi_\varepsilon$ denotes an approximation of the identity, as described above. Moreover, the limiting function $
\displaystyle v_0 := \lim_{\varepsilon \to 0} v_\varepsilon $
is Lipschitz continuous and satisfies a one-phase free boundary problem.

In addition, \cite{BS2025} investigated the free boundary problem
\begin{equation}\label{FBP}
\begin{cases}
\Delta_g v = f & \text{in } \{v>0\} \cap B_1^{+},\\[4pt]
|\nabla v^{+}| \le \mathcal{Q} & \text{along } \mathscr{F}_0(v, B_1^{+}),\\[4pt]
v = g & \text{on } B_1'.
\end{cases}
\end{equation}
Here $0 \le \mathcal{Q}$ is a bounded, continuous function defined along the free boundary of $v$, namely
\[
\mathscr{F}_0(v, B_1^{+}) := \partial\{v>0\} \cap B_1^{+}.
\]
Furthermore, the authors impose a natural condition on $g$ to ensure that $g^{+}$ remains of class $C^{1,\alpha}$ on the fixed boundary, namely a degenerate phase transition condition of the form \eqref{DegPhasTran}. Within this framework, they derived up-to-the-boundary Lipschitz estimates for solutions to \eqref{FBP}.


\medskip

\medskip

Now, we must stress that Fernandez-Real and Guen's manuscript \cite{FerReal-Gruen2025} makes significant advances in the understanding of the boundary regularity of minimizers of the \textbf{one-phase Bernoulli problem}, represented by the Alt-Caffarelli functional. The principal contribution lies in establishing \textbf{continuity up to the boundary} for minimizers under minimal regularity assumptions on the boundary data, notably for continuous and H\"{o}lder-continuous boundary conditions.
Specifically, the authors proved that for open domains $ \mathrm{D} \subset \mathbb{R}^{d} $ that are either convex or possess a sufficiently small Lipschitz constant, minimizers of the functional 
$$
\displaystyle \mathcal{J}_{\Lambda}(u, \mathrm{D})  = \int_{\mathrm{D}} (|\nabla u|^2 + \Lambda |\{u>0\} \cap \mathrm{D}|)dx
$$ with continuous boundary data are themselves continuous up to $ \partial \mathrm{D} $. This result enhances previous boundary regularity theorems, which often required smoother domain assumptions or stricter boundary regularity of the data (cf. \cite{Caffarelli1977}).

The methodology employed hinges on a combination of classical potential theory, comparison principles, and barrier arguments tailored to the free boundary context. Notably, the authors adapt and extend techniques from the classical theory of elliptic partial differential equations (see \cite{GiTr}) to the free boundary setting, overcoming challenges arising from the inherent nonlinearity and the variational structure of the problem.
A key aspect of this study is the demonstration that the boundary regularity of minimizers does not solely rely on smooth boundary data, but persists under mere continuity assumptions, provided the domain satisfies certain geometric regularity conditions. This aligns with and extends the scope of prior results on boundary behaviour in free boundary problems, such as \cite{AltCaffarelli1981} and \cite{JerinsonSavin2015}.

Implications of these results are twofold: first, they assure the well-posedness of the Dirichlet problem for continuity data in broader geometric contexts; second, they facilitate the analysis of generic regularity and uniqueness phenomena for families of solutions with continuous boundary conditions (\cite{Fernandez-RealYu2023}). In conclusion, the work \cite{FerReal-Gruen2025} significantly broadens the understanding of how minimal regularity conditions on the boundary data influence the regularity of solutions, thereby contributing critical insights to the theory of free boundary problems and their applications.

\medskip

Finally, we must recall that the paper by Bayrami \textit{et al.} \cite{BFS24}  where the authors established significant advances in the boundary regularity theory for almost-minimizers of variational functionals involving the $p$-Laplacian, specifically those of the form
\begin{equation}
\mathcal{J}(\mathbf{v}; \mathrm{D}) = \int_{\mathrm{D}} \sum_{i=1}^m |\nabla v_i|^p + \lambda \chi_{\{|\mathbf{v}|>0\}} \, dx,
\end{equation}
where $\mathbf{v} = (v_1, \ldots, v_m)$ with $v_i \in W^{1,p}(\mathrm{D})$, posed on a bounded Lipschitz domain $\mathrm{D} \subset \mathbb{R}^n$, for $1 < p < \infty$ and a fixed parameter $\lambda > 0$.

Concerning the more delicate boundary behaviour, the authors concentrate on the regime $p \geq 2$, which encompasses many physically relevant models. In this setting, they develop an \textbf{alternative method}, distinct from classical strategies.
Their approach unfolds through the following components:
\begin{itemize}
\item[\checkmark] A careful flattening of the boundary, reducing the analysis to a half-space configuration.
\item[\checkmark] Boundary gradient estimates obtained by adapting interior regularity arguments to the boundary, leveraging the Lipschitz geometry of the domain.
\end{itemize}
This systematic program yields the result that the vectorial almost-minimizer $\mathbf{v}$ is \textbf{Lipschitz continuous up to the boundary} for $p \geq 2$. It also furnishes an \textbf{alternative proof} of interior Lipschitz regularity that circumvents more involved blow-up procedures, offering a streamlined and constructive argument for regularity near the boundary.

\medskip

The attainment of boundary Lipschitz regularity for almost-minimizers substantially enriches the understanding of free boundary problems associated with Alt-Caffarelli type functionals in the context of weakly coupled vector-valued almost-minimizers in Orlicz spaces. It not only confirms the mathematical well-posedness of these profiles but also ensures their stability and robustness in physical applications involving such energies (cf. \cite{FSS}).
Therefore, motivated by the previous contributions, the development of an alternative boundary regularity framework for $G \in \mathcal{G}(\delta,g_0)$ with $\delta > 1$ opens avenues for extending these techniques to broader classes of nonlinear degenerate elliptic problems (e.g., weakly coupled vectorial almost-minimizers in Musielak-Orlicz and generalized Orlicz spaces), particularly in settings where boundary behaviour plays a decisive role in the qualitative analysis of solutions and free boundaries.

\section{On the Orlicz and Orlicz-Sobolev spaces}\label{Sec2}

Here we recall some properties of the Orlicz spaces, which can be found in \cite{A, HH, RR}. In what follows, we begin by recalling the basic definitions and some standard properties that will be frequently used throughout the paper.

The notion of an $\mathcal{N}$-function has already been introduced in the Introduction (see Definition \ref{Def-N-function}). We would like to recall that, by the convexity of any $\mathcal{N}$-function $G$, for every $t \in (0,+\infty)$ one has the following property:
	\begin{equation}\label{ine}
		G(\alpha t) \le \alpha G(t) \ \mbox{for} \ \alpha \in [0,1] \quad \mbox{and} \quad G(\tilde{\alpha} t) > \tilde{\alpha} G(t) \ \mbox{for} \ \tilde{\alpha}>1.
	\end{equation} 
    
\begin{definition}[{\bf Orlicz spaces}] The Orlicz space associated to $\mathcal{N}$-function $G$ and an open $U\subset \mathbb{R}^n$ is defined by
	$$
	L_G(U;\mathbb{R}^m)= \left\{{\bf u} \in L^1_{loc}(U;\mathbb{R}^m) \;: \;
	\int_U G \left(\frac{\vert {\bf u}(x)\vert}{\gamma}\right)\, dx<+\infty, \ \mbox{for some}~ \gamma>0\right\},
	$$
and
	$$
	|{\bf u}|_{L_G(U;\mathbb{R}^m)}=\inf \left\{\gamma>0\;:\;\int_U
	G \left(\frac{\vert {\bf u}(x)\vert}{\gamma} \right)\, dx\leq 1 \right\},
	$$
defines a norm (the Luxemburg norm) on $L_G(U;\mathbb{R}^m)$ and turns this space into a Banach space. In the study of the Orlicz space $L_G(U;\mathbb{R}^m)$, we denote by $K_G(U;\mathbb{R}^m)$ the Orlicz class as the set below
	$$K_G(U;\mathbb{R}^m) \coloneqq \left\{{\bf u} \in L^1_{loc}(U;\mathbb{R}^m)\;: \; \int_U G \big({\vert {\bf u}(x)\vert}\big)\, dx<+\infty\right\}.$$
In this case, we have
	$$K_G(U;\mathbb{R}^m) \subset L_G(U;\mathbb{R}^m),$$
where $L_G(U;\mathbb{R}^m)$ is the smallest subspace containing $K_G(U;\mathbb{R}^m)$.
\end{definition}

An important function in the study of the properties of the Orlicz space $L_G(U;\mathbb{R}^m)$ is the conjugate function of $G$, labelled $\widetilde{G}$, which is an $\mathcal{N}$-function defined by the Legendre transform of $G$, that is,
    $$\widetilde{G}(t) \coloneqq \sup \{st - G(s) \; \; : \;\; s \in [0,\infty)\}, \quad t \in [0,\infty).$$
Using the function $\widetilde{G}$, we have the Young inequality
    \begin{equation}\label{young}
        st \le \varepsilon G(s) + c(\varepsilon)\widetilde{G}(s),
    \end{equation}
for any $s,t\ge 0$ and $\varepsilon>0$.

\begin{definition}[{\bf $\Delta_2$ and $\nabla_2-$condition}]
An $\mathcal{N}$-function $G$ satisfies the  $\Delta_2$-condition (shortly, $G\in \Delta_2$), if there exist $k_G>0$ and $t_0\geq0$  such that
	\begin{equation}\label{defD2}
		G(2t)\leq  k_GG(t),~  t\geq t_0.
	\end{equation}
Similarly, we say that an $\mathcal{N}$-function $G$ satisfies the  $\nabla_2$-condition (shortly, $G\in \nabla_2$), if there exist $l_G>1$ and $t_0\geq0$  such that
	\begin{equation}\label{defn2}
		G(t)\leq  \dfrac{1}{2l_G}G(l_Gt),~  t\geq t_0.
	\end{equation}
\end{definition}
It is possible to prove that if $G \in \Delta_2$ then
	$$K_G(U;\mathbb{R}^m) = L_G(U;\mathbb{R}^m).$$

The corresponding Orlicz-Sobolev space is defined by
    $$W^{1, G}(U;\mathbb{R}^m) = \left\{ {\bf u} \in L_{G}(U;\mathbb{R}^m) \ :\ \frac{\partial u_j}{\partial x_{i}} \in L_{G}(U;\mathbb{R}^m), \quad \substack{i = 1, \dots, n\\ j=1,\dots, m}\right\},$$
endowed with the norm
    $$\Vert {\bf u} \Vert_{W^{1,G}(U;\mathbb{R}^m)} =   \vert\nabla {\bf u}\vert _{L_G(U;\mathbb{R}^{nm})} +  \vert {\bf u}\vert_{L^G(U;\mathbb{R}^m)}.$$
The space $W^{1,G}(U;\mathbb{R}^m)$ endowed with $\|\cdot\|_{1,G}$ is always a Banach space. Furthermore, this space is reflexive and separable if, and only if, $G\in\Delta_2\cap\nabla_2$. The subspace $W_0^{1,G}(U;\mathbb{R}^m)$ is defined as the weak$^*$ closure of $C_0^\infty(U;\mathbb{R}^m)$ in $W^{1,G}(U;\mathbb{R}^m)$; under $\Delta_2$-condition this space coincides with the norm closure of $C_0^\infty(U;\mathbb{R}^m)$, which provides a more practical characterization. For simplicity, when $m=1$ we will denote $L_G(U;\mathbb{R})$, $W^{1,G}(U;\mathbb{R})$, and $W_0^{1,G}(U;\mathbb{R})$ solely as $L_G(U)$, $W^{1,G}(U)$, and $W_0^{1,G}(U)$, respectively. 

Next, we turn our attention to the Lieberman class introduced in Definition \ref{defclasses}, which is known to enjoy a number of remarkable analytical features. This class has been extensively studied in the literature and plays a central role in regularity theory for problems with nonstandard growth. In particular, Lieberman \cite{L}, Martínez and Wolanski \cite{MW08}, and Fukagai, Ito, and Narukawa \cite{FIN} have established several important results concerning this class, which we summarize in the following statement.

\begin{statement}[{\bf Properties of the class $\mathcal{G}(\delta,g_0)$}]\label{Statement} Suppose $G\in \mathcal{G}(\delta,g_0)$. The following properties hold:

\begin{description}
	\item[$(g_1)$] $\min\left\{s^{\delta},s^{g_0}\right\}g(t) \le g(st) \le \max\left\{s^{\delta},s^{g_0}\right\}g(t)$;
	
	\item[$(g_2)$] $G$ is a convex and $C^2$ function;
	
	\item[$(g_3)$] $\dfrac{sg(s)}{g_0+1} \le G(s) \le sg(s),$ for any $s\ge 0$;

    \item[$(G_1)$]  Defining $\xi_0(t)\coloneqq \min\{t^{\delta+1},t^{g_0+1}\}$ and $\xi_1(t)\coloneqq \max\{t^{\delta+1},t^{g_0+1}\}$, for any $s,t\ge0$ and $u\in L_G(U)$ hold
        $$\xi_0(s)\dfrac{G(t)}{g_0+1} \le G(st) \le (g_0+1)\xi_1(s)G(t), \quad \mbox{and} \quad \xi_0(|u|_{L_G}) \le \int_{U}G(|u|)\ dx \le \xi_1(|u|_{L_G});$$

    \item[$(G_2)$]  $G(s)+G(t) \le G(s+t) \le 2^{g_0}(g_0+1)\big(G(s) + G(t)\big)$ for any $s,t\ge0$;

    \item[$(G_3)$] $\widetilde{G}\big(g(s)\big) \le g_0G(s)$ for all $s\ge0$;

    \item[$(G_4)$] Defining $\zeta_0(t)\coloneqq \min\left\{t^\frac{\delta+1}{\delta},t^\frac{g_0+1}{g_0}\right\}$ and $\zeta_1(t)\coloneqq \max\left\{t^\frac{\delta+1}{\delta},t^\frac{g_0+1}{g_0}\right\}$, for any $s,t\ge0$ and $u\in L_{\widetilde{G}}(U)$ hold
        $$\zeta_0\left(|u|_{L_{\widetilde{G}}}\right) \le \int_{U}\widetilde{G}(|u|)\ dx \le \zeta_1\left(|u|_{L_{\widetilde{G}}}\right).$$
\end{description}
\end{statement}

Therefore, any $G\in\mathcal{G}(\delta,g_0)$ satisfies the $\Delta_2$-condition. Furthermore, a crucial implication of \eqref{Ga} is that
\begin{equation}\label{UE}
	\min\{\delta,1\} \dfrac{g(|x|)}{|x|} |\xi|^2 \le \mathfrak{a}_{i,j} \xi_i \xi_j \le \max\{g_0,1\} \dfrac{g(|x|)}{|x|} |\xi|^2,
\end{equation}
for any $x,\xi \in \mathbb{R}^N$ and $i,j =1,...,n$, where $\mathfrak{a}_{i,j} = \frac{\partial \mathfrak{A}_i}{\partial x_j}$ with $\mathfrak{A}(x)=g(|x|)\frac{x}{|x|}$. This inequality means that the equation
\begin{equation}\label{G-harmonic}
	-\mbox{div} \left(g(|\nabla {\bf v}|)\dfrac{\nabla {\bf v}}{|\nabla {\bf v}|}\right) = 0, \quad \mbox{in} \ U,
\end{equation}
is uniformly elliptic for $\frac{g(|x|)}{|x|}$ bounded and bounded away from zero, for any domain $U\subset \mathbb{R}^n$. Moreover, if ${\bf v} \in W^{1,G}(U;\mathbb{R}^m)$ satisfies \eqref{G-harmonic}, we say that ${\bf v}$ is $g$-harmonic in $U$.

Additionally, we have the following general version of Poincar\'{e}'s inequality (see, for instance, \cite[Theorem 7]{DE}, or in an even more general version \cite[Lemma 4.1]{BB}). 
    \begin{lemma}\label{Poincare}
        Let $G$ be any function on the class $\mathcal{G}(\delta,g_0)$. Then, there exists $\theta=\theta(n,\delta,g_0)\in (0,1)$ such that
            \begin{equation*}
                \intav{U} G\left(\dfrac{|v-(v)_U|}{\mbox{diam}\ ( U)}\right) dx \le \mathrm{C}\left(\intav{U} \Big(G(|\nabla v|)\Big)^\theta dx\right)^{\frac{1}{\theta}},
            \end{equation*}
        where $\mathrm{C}=\mathrm{C}(n,\delta,g_0)>0$, whenever $v \in W^{1,G}(U)$. Moreover, the above estimate still holds with $v-(v)_U$ replaced by $v$, if $v \in W_0^{1,G}(U)$.
    \end{lemma}
In the previous result, as well as throughout the paper, we have the following notation
    $$(v)_U \coloneqq \intav{U} v\ dx \coloneqq \dfrac{1}{|U|}\int_{U} v\ dx.$$

Now, for a given $\mathcal{N}$-function $G\in \mathcal{G}(\delta,g_0)$, define the map $\mathbf{V}_G:\mathbb{R}^n \to \mathbb{R}^n$ by
    \begin{equation}\label{defVg}
        \mathbf{V}_G(z) \coloneqq \left(\dfrac{G^{\prime}(|z|)}{|z|}\right)^{\frac{1}{2}}z.
    \end{equation}
Hence, we have the following properties (see, for instance, \cite{BB} in a double-phase structure)
    \begin{equation}\label{SSMI}
        \mathrm{C}|\mathbf{V}_G(z_1) - \mathbf{V}_G(z_2)|^2 + \dfrac{G^{\prime}(|z_1|)}{|z_1|} \ z_1\cdot (z_2 - z_1) \le G(|z_2|) - G(|z_1|), \quad (z_1 \neq 0)
    \end{equation}
and if $\delta\ge 1$ we have
    \begin{equation}\label{SSMII}
        G(|z_1-z_2|) \le \mathrm{C}|\mathbf{V}_G(z_1)-\mathbf{V}_G(z_2)|^2,
    \end{equation}
for some $\mathrm{C}=\mathrm{C}(\delta,g_0)>0$. More precisely, we have the following control of the monotonicity of $G$.

\begin{lemma}[{\bf see \cite[Theorem 1]{BC}}]\label{monotonicity}
    Suppose that $G \in \mathcal{G}(\delta,g_0)$, with $\delta>0$. Then, for any $z_1,z_2 \in \mathbb{R}^n$, there exist constants $c_1,c_2>0$ depending only on $\delta$ and $g_0$ such that
        $$c_1 \left(\dfrac{G(|z_1|)}{|z_1|^2}z_1 - \dfrac{G(|z_2|)}{|z_2|^2}z_2\right)\big(z_1 -z_2\big) \le |\mathbf{V}_G(z_1)-\mathbf{V}_G(z_2)|^2 \le c_2\left(\dfrac{G(|z_1|)}{|z_1|^2}z_1 - \dfrac{G(|z_2|)}{|z_2|^2}z_2\right)\big(z_1 -z_2\big).$$
\end{lemma}

\begin{lemma}\label{growthg}
    Suppose that $G \in \mathcal{G}(\delta,g_0)$, with $\delta\ge 1$. Then, for any $z_1,z_2 \in \mathbb{R}^n\setminus\{\mathbf{0}\}$ there exists $\mathrm{C}=\mathrm{C}(n,g_0)>0$
        $$\left|\dfrac{g(|z_1|)}{|z_1|}z_1 - \dfrac{g(|z_2|)}{|z_2|}z_2\right| \le \mathrm{C}\dfrac{g\big(|z_1|+|z_2|\big)}{|z_1|+|z_2|}|z_1-z_2|.$$
\end{lemma}
\begin{proof}
    For simplicity, define $\mathscr{A}(z) \coloneqq \frac{g(|z|)}{|z|}z$ for $z \in \mathbb{R}^n\setminus\{\mathbf{0}\}$. A straightforward computation gives
        $$D_z\mathscr{A}(z) = \dfrac{g(|z|)}{|z|}\text{Id}_n + \left(\dfrac{g'(|z|)}{|z|^2} - \dfrac{g(|z|)}{|z|^3}\right)z \otimes z,$$
    where $\text{Id}_n$ denotes the identity matrix and $\otimes$ the standard tensor product. Let $\gamma(s) \coloneqq z_2 + s(z_1 - z_2)$ for $s \in [0,1]$. By the fundamental theorem of calculus,
        \begin{eqnarray*}
            \mathscr{A}(z_1) - \mathscr{A}(z_2) &=& \int_0^1 \dfrac{d}{ds}\Big(\mathscr{A}(\gamma(s))\Big) \ ds \\
            &=& \int_0^1 D_z \mathscr{A}\big(\gamma(s)\big) \ (z_1-z_2)\ ds \\
            &=& (z_1-z_2) \int_0^1 \left[\dfrac{g(|\gamma(s)|)}{|\gamma(s)|}\text{Id}_n + \left(\dfrac{g'(|\gamma(s)|)}{|\gamma(s)|^2} - \dfrac{g(|\gamma(s)|)}{|\gamma(s)|^3}\right)\gamma(s) \otimes \gamma(s)\right]\ ds.
        \end{eqnarray*}
    Hence, by \eqref{Ga} and the monotonicity of the map $t \mapsto \frac{g(t)}{t}$ (which holds for $\delta \ge 1$), we obtain
        \begin{eqnarray*}
            |\mathscr{A}(z_1)-\mathscr{A}(z_2)| &\le& \mathrm{C}(n)|z_1-z_2| \int_0^1 \left[\dfrac{g(|\gamma(s)|)}{|\gamma(s)|} + g'(|\gamma(s)|) + \dfrac{g(|\gamma(s)|)}{|\gamma(s)|}\right]\ ds \\
            &\le& \mathrm{C}(n,g_0) |z_1-z_2| \int_0^1 \dfrac{g(|\gamma(s)|)}{|\gamma(s)|}\ ds \\
            &\le& \mathrm{C}(n,g_0) \dfrac{g(|z_1|+|z_2|)}{|z_1|+|z_2|} |z_1-z_2|.
        \end{eqnarray*}
    This completes the proof.
\end{proof}

\subsection{Auxiliary estimates for non-autonomous operator}\label{Sec2.1}

In this subsection, we establish some results concerning harmonic solutions for a non-autonomous operator that play an important role in our subsequent analysis. Before proceeding, let us recall a classical iteration lemma that will be instrumental in achieving our main objectives.

\begin{lemma}[{see \cite[Lemma 5.13]{G}}]\label{interationlemma}
    Let $\phi:\mathbb{R}^+ \to \mathbb{R}^+$ be a non-negative and non-decreasing function satisfying
        $$\phi(\rho) \le A\left[\left(\dfrac{\rho}{R}\right)^\tau+\varepsilon\right]\phi(R) + BR^\theta,$$
    for some $A,B,\tau,\theta>0$ with $\tau>\theta$ and for all $0<\rho\le R\le R_0$, where $R_0>0$ is given. Then, there exist constants $\varepsilon_0= \varepsilon_0(A,\tau,\theta)>0$ and $c=c(A,\tau,\theta)>0$ such that if $\varepsilon>0$ given in the above inequality satisfies $\varepsilon\le \varepsilon_0$, we have
        $$\phi(\rho) \le c\left[\left(\dfrac{\rho}{R}\right)^\theta\phi(R)+B\rho^\theta\right],$$
    for any $0<\rho \le R \le R_0$.
\end{lemma}

From now on, we will use, for $x_0 \in \mathbb{R}^{n-1}\times \{0\}$, the following notation
    $$B^+_\rho(x_0) \coloneqq \{x \in \mathbb{R}^n \; \; : \;\; x_n>0 \; \mbox{and} \: |x-x_0| < \rho \}, \quad D^+_\rho(x_0) \coloneqq \{x \in \mathbb{R}^n \; \; : \;\; x_n=0 \; \mbox{and} \: |x-x_0| < \rho \},$$
as the half ball and the flat boundary, respectively. For simplicity of notation, we write $B^+_\rho(0) = B^+_\rho$ and $D^+_\rho(0) = D^+_\rho$. Further, we introduce the vector field $\mathfrak{a}:B_R^+\times\mathbb{R}^n \to \mathbb{R}^n$ by
    \begin{equation}\label{defvecfield-a}
        \mathfrak{a}(x,\eta)\coloneqq\dfrac{g\big(|\eta\cdot A(x)|\big)}{|\eta\cdot A(x)|}\eta\cdot A(x) A^T(x), \quad \mbox{for any} \ (x,\eta) \in B_R^+\times\mathbb{R}^n,
    \end{equation}
where $G\in \mathcal{G}(\delta,g_0)$, $A : \mathbb{R}^n \to \mathbb{R}^n$ is a matrix-valued map satisfying the uniform ellipticity bounds
    \begin{itemize}
        \item $\sqrt{\frac{1}{2}}|w| \le |A(x)w| \le \sqrt{\frac{3}{2}}|w|$ for any $w\in\mathbb{R}^n$ and $x \in B_{\sqrt{2}R}$.
    \end{itemize}
and $A^T$ denotes the transpose of $A$. Further, we say that $u\in W^{1,G}(B_R^+)$ is $\mathfrak{a}$-harmonic if
    $$-\mbox{div} \ \mathfrak{a}(x,\nabla u)=0,\quad \mbox{in}\ B_R^+,$$
in the weak sense.

The next three remarks collect properties of the vector field $\mathfrak{a}$ that are well known in the literature in the autonomous setting. These properties will play a crucial role in the proof of Lemma~\ref{g-harmonic_control} below. Although their validity in the present framework follows from standard arguments, we briefly outline the main ideas behind their proofs for the sake of completeness and to clarify how the dependence on the matrix $A(x)$ enters the analysis.

\begin{remark}[{\bf Odd reflection}]\label{Remark-odd_reflection}

Let $w \in W^{1,G}(B_R^+)$ be a minimizer of the functional
\[
u \mapsto \int_{B_R^+} G\big(|\nabla u\, A(x)|\big)\,dx
\]
over the admissible class
\[
\mathcal{A} := \bigl\{ u \in W^{1,G}(B_R^+) : u = v-h \text{ on } \partial B_R^+ \bigr\}.
\]
Then $w$ is a weak solution of the associated Euler--Lagrange equation
\[
\int_{B_R^+} \mathfrak{a}(x,\nabla w)\cdot \nabla \varphi \, dx = 0
\quad \text{for all } \varphi \in W^{1,G}_0(B_R^+),
\]
where $\mathfrak{a}(\cdot,\cdot)$ is defined in \eqref{defvecfield-a}. Moreover, by construction, $w=0$ on $D_R^+$.

Define $\widetilde w$ as the odd extension of $w$ to $B_R$.
Since $w=0$ on $D_R^+$, it follows that $\widetilde w \in W^{1,G}(B_R)$.
Assume that the matrix field $A(x)$ is sufficiently regular and even with
respect to the reflection across $\{x_n=0\}$, so that the operator
$\mathfrak{a}(x,\eta)$ is invariant under this symmetry. Accordingly, define
the even extensions
\[
\widetilde A(x',x_n) :=
\begin{cases}
A(x',x_n), & x_n>0,\\[0.3em]
A(x',-x_n), & x_n<0,
\end{cases}
\qquad
\widetilde{\mathfrak{a}}(x,\xi)
:= \mathfrak{a}(x',|x_n|,\xi).
\]

Let $\phi \in C_c^\infty(B_R)$ be arbitrary and decompose it as
\[
\phi = \phi_{\mathrm{odd}} + \phi_{\mathrm{even}},
\]
with respect to the reflection $x_n \mapsto -x_n$. Testing the weak formulation
in $B_R^+$ with $\phi_{\mathrm{odd}}|_{B_R^+}$ and exploiting the odd symmetry of
$\widetilde w$, we infer that
\[
\int_{B_R} \widetilde{\mathfrak{a}}(x,\nabla \widetilde w)\cdot \nabla \phi \, dx = 0.
\]
Consequently, $\widetilde w$ is a weak solution of
\[
-\operatorname{div}\ \widetilde{\mathfrak{a}}(x,\nabla \widetilde w)=0
\quad \text{in } B_R,
\]
that is, $\widetilde w$ is $\widetilde{\mathfrak{a}}$-harmonic in $B_R$. Since $\widetilde w = w$ in $B_R^+$, the odd extension of $w$ is
$\widetilde{\mathfrak{a}}$-harmonic in the entire ball $B_R$.
\end{remark}

\begin{remark}[{\bf Energy estimates}] \label{Remark2.2}

Let $v$ be a minimizer of the functional 
\[
u \mapsto \int_{B_R} G\big(|\nabla u\, A_0(x)|\big)\,dx,
\]
where $G$ is an $\mathcal N$-function satisfying the $\Delta_2$-condition and $A_0\in C^0(\Omega;\mathbb{R}^{n\times n})$ is uniformly bounded and elliptic. Then, $v$ satisfies
\[
\int_{B_\rho} G\big(|\nabla v\, A_0(x)|\big)\,dx
\le
\mathrm{C}_0\left(\frac{\rho}{R}\right)^n
\int_{B_R} G\big(|\nabla  v\, A_0(x)|\big)\,dx,
\qquad 0<\rho\le R,
\]
with $\mathrm{C}_0=\mathrm{C}_0\!\left(n,G,\|A_0\|_{C^0(\Omega)}\right)>0$. Indeed, to prove this remark, let us split into three steps. Throughout the proof, constants may vary from line to line but depend only on $n$, $G$, and $\| A_0\|_{C^0(\Omega)}$.

\medskip
\noindent
{\bf Step 1:} \emph{Caccioppoli-type inequality.}
Fix $0<\rho<R$ and choose a cutoff function $\eta\in C_c^\infty(B_R)$ such that
\[
\eta\equiv 1 \ \text{in } B_\rho, 
\qquad 
0\le \eta\le 1,
\qquad 
|\nabla\eta|\le \frac{\mathrm{C}}{R-\rho}.
\]
Testing the minimality of $\widetilde w$ with the comparison map
\[
u := v - \eta\bigl(v-(v_{B_R})\bigr)
\]
and using the convexity of $G$ together with standard truncation arguments, one obtains the Caccioppoli inequality
\[
\int_{B_\rho} G\big(|\nabla v\, A_0(x)|\big)\,dx
\le
\mathrm{C} \int_{B_R}
G\!\left(\frac{|v-(v)_{B_R}|}{R-\rho}\right)\,dx.
\]
Such inequalities are classical for minimizers with Orlicz growth; see
Acerbi--Fusco~\cite{AcerbiFusco1984} or
Diening--Harjulehto--H\"ast\"o--R\r{u}\v{z}i\v{c}ka~\cite[Chapter~6]{DieningEtAl2011}.

\medskip
\noindent
{\bf Step 2:} \emph{Orlicz--Poincar\'e inequality.}
Since $G$ satisfies the $\Delta_2$-condition, the Orlicz--Poincar\'e inequality yields
\[
\int_{B_R}
G\!\left(\frac{|v-(v)_{B_R}|}{R-\rho}\right)\,dx
\le
\mathrm{C} \int_{B_R} G\big(|\nabla v|\big)\,dx,
\]
see \cite[Theorem~8.7.2]{DieningEtAl2011}.
Furthermore, using the boundedness and ellipticity of $A_0$, we infer
\[
\int_{B_R} G\big(|\nabla v|\big)\,dx
\le
\mathrm{C} \int_{B_R} G\big(|\nabla v\, A_0(x)|\big)\,dx.
\]
Combining the above estimates gives
\[
\int_{B_\rho} G\big(|\nabla v\, A_0(x)|\big)\,dx
\le
\mathrm{C} \int_{B_R} G\big(|\nabla v\, A_0(x)|\big)\,dx.
\]

\medskip
\noindent
{\bf Step 3:} \emph{Scaling and decay.}
 Define the rescaled function $v_R(x):= v(Rx)$ for $x\in B_1$.
A change of variables yields
\[
\int_{B_\rho} G\big(|\nabla v\, A_0(x)|\big)\,dx
=
R^n \int_{B_{\rho/R}}
G\big(|\nabla v_R(x)\, A_0(Rx)|\big)\,dx.
\]
Applying the previous estimate at scale $1$ and using the homogeneity of Lebesgue measure, we obtain
\[
\int_{B_{\rho/R}}
G\big(|\nabla v_R\, A_0(Rx)|\big)\,dx
\le
\mathrm{C}\left(\frac{\rho}{R}\right)^n
\int_{B_1}
G\big(|\nabla v_R\, A_0(Rx)|\big)\,dx.
\]
Rescaling back gives
\[
\int_{B_\rho} G\big(|\nabla v\, A_0(x)|\big)\,dx
\le
\mathrm{C}_0\left(\frac{\rho}{R}\right)^n
\int_{B_R} G\big(|\nabla v\, A_0(x)|\big)\,dx,
\qquad 0<\rho\le R,
\]
with $\mathrm{C}_0=\mathrm{C}_0\!\left(n,G,\|\ A_0\|_{C^0(\Omega)}\right)>0$. 
\end{remark}

\begin{remark}[{\bf Lieberman-type energy decay}]\label{Remark-energy_decay}

Let $0<\rho\le R$, and let $\widetilde w\in W^{1,G}(B_R)$ denote the odd extension
of $w$ constructed above. As shown previously, $\widetilde w$ is a weak
solution of
\[
-\operatorname{div}\ \widetilde{\mathfrak a}(x,\nabla \widetilde w)=0
\quad \text{in } B_R,
\]
where $\widetilde{\mathfrak a}$ is even with respect to the reflection
$x_n \mapsto -x_n$ and satisfies the standard $G$-growth and ellipticity
assumptions inherited from $\mathfrak a$, together with the regularity of
$A(x)$.

It is a classical consequence of the minimality of $\widetilde w$, equivalently,
of its $\widetilde{\mathfrak a}$-harmonicity, that $\widetilde w$ is a local
minimizer of the functional
\[
u \mapsto \int_{B_R} G\big(|\nabla u\,\widetilde A(x)|\big)\,dx
\]
with respect to compact perturbations. Consequently, by standard comparison
and scaling arguments for minimizers with Orlicz growth (cf.\ Caccioppoli-type
inequalities and energy-type decay estimates - see Remark \ref{Remark2.2}), there exists a
constant $\mathrm{C}_0=\mathrm{C}_0\!\left(n,G,\|A\|_{C^0(\Omega)}\right)>0$ such
that
\[
\int_{B_\rho} G\big(|\nabla \widetilde w\,\widetilde A(x)|\big)\,dx
\le
\mathrm{C}_0\left(\frac{\rho}{R}\right)^n
\int_{B_R} G\big(|\nabla \widetilde w\,\widetilde A(x)|\big)\,dx,
\qquad 0<\rho\le R.
\]

Since $\widetilde w$ is the odd extension of $w$ and $\widetilde A$ is the even
extension of $A$, we have
\[
\int_{B_\rho} G\big(|\nabla \widetilde w\,\widetilde A(x)|\big)\,dx
=
2\int_{B_\rho^+} G\big(|\nabla w\,A(x)|\big)\,dx \quad \text{and}
\quad
\int_{B_R} G\big(|\nabla \widetilde w\,\widetilde A(x)|\big)\,dx
=
2\int_{B_R^+} G\big(|\nabla w\,A(x)|\big)\,dx.
\]
Combining these identities yields
\[
\int_{B_\rho^+} G\big(|\nabla w\,A(x)|\big)\,dx
\le
\mathrm{C}_0\left(\frac{\rho}{R}\right)^n
\int_{B_R^+} G\big(|\nabla w\,A(x)|\big)\,dx,
\qquad 0<\rho\le R,
\]
which establishes the claimed energy decay estimate.
\end{remark}

\begin{lemma}\label{g-harmonic_control}
    Consider $G \in \mathcal{G}(\delta,g_0)$, with $\delta>0$ and $\mathfrak{a}(\cdot,\cdot)$ is defined \eqref{defvecfield-a}. Assume that $v \in W^{1,G}(B^+_R)$ is a weak solution of
        \begin{equation}\label{a-harmonic}
            \left\{ \begin{array}{rccl}
            -\mbox{div}\ \mathfrak{a}(x,\nabla v) &=&0,& \mbox{in}\ B^+_R; \\
            v&=&h,& \mbox{in}\ D^+_R,
        \end{array}\right.
        \end{equation}
    for a given $h \in W^{1,q}(B^+_R)$, with $q > g_0+1$. Then, for all $\rho \in (0,R]$ we have
        $$\int_{B^+_\rho} G(|\nabla v \,A(x)|)\ dx \le C \left( \left(\dfrac{\rho}{R}\right)^{n\left(1-\frac{g_0+1}{q}\right)}\int_{B^+_R}G(|\nabla v{\, A(x)}|)\ dx + \rho^{n\left(1-\frac{g_0+1}{q}\right)}\xi_1\left(\|\nabla h\|_{L^q(B_R^+)}\right)\right),$$
    where $\mathrm{C}=\mathrm{C}\big(\delta,g_0,n,\|A\|_{C^0},G(1)\big)>0$, and the function $\xi_1$ is defined in $(G_1)$.
\end{lemma}

\begin{proof}
    To begin with, let us denote by $w$ the minimizer of the functional
        $$u \mapsto \int_{B^+_R} G(|\nabla u{\, A(x)}|)\ dx,$$
    in the set $\{u \in W^{1,G}(B^+_R) \; : \; u = v-h \ \mbox{on} \ \partial B^+_R\}$. Thus, $w$ is $\mathfrak{a}$-harmonic in $B^+_R$ and $w=0$ on $D^+_R$. Therefore, by Remark \ref{Remark-energy_decay}
        \begin{equation}\label{Gw}
            \int_{B^+_\rho} G(|\nabla w{\, A(x)}|)\ dx \le \mathrm{C} \left(\dfrac{\rho}{R}\right)^n \int_{B^+_R} G(|\nabla w{\, A(x)}|)\ dx, \quad \mbox{for all}\ 0<\rho \le R.
        \end{equation}
    
    \noindent Since $v-w-h \in W^{1,G}(B^+_R)$, with $v-w-h =0$ on $D^+_R$, we may ensure by the $\mathfrak{a}$-harmonicity of $v$ and $w$ and Lemma \ref{growthg} that
        \begin{eqnarray}\label{213}
            I_A&\coloneqq&\int_{B^+_R}\left(\dfrac{g(|\nabla v{\, A(x)}|)}{|\nabla v{\, A(x)}|}\nabla v{\, A(x)} - \dfrac{g(|\nabla w{\, A(x)}|)}{|\nabla w{\, A(x)}|}\nabla w{\, A(x)}\right){A^T(x)}\big(\nabla v - \nabla w\big)\ dx \nonumber \\
            &=& \int_{B^+_R} \left(\dfrac{g(|\nabla v{\, A(x)}|)}{|\nabla v{\, A(x)}|}\nabla v{\, A(x)} - \dfrac{g(|\nabla w{\, A(x)}|)}{|\nabla w{\, A(x)}|}\nabla w{\, A(x)}\right){A^T(x)}\nabla h\ dx \nonumber \\
            &\le& \mathrm{C}(n,g_0) \int_{B^+_R} \dfrac{g(|\nabla v{\, A(x)}|+|\nabla w{\, A(x)}|)}{|\nabla v{\, A(x)}|+|\nabla w{\, A(x)}|}|\nabla v{\, A(x)}-\nabla w{\, A(x)}|{|A^T(x)}\ \nabla h|\ dx.
        \end{eqnarray}
    Now, let us estimate the right-hand side. Note that the Young inequality \eqref{young}, $(G_3)$, and properties of $A(x)$ yield
        \begin{eqnarray}\label{214}
            &&\int_{B^+_R} \dfrac{g(|\nabla v{\, A(x)}|+|\nabla w{\, A(x)}|)}{|\nabla v{\, A(x)}|+|\nabla w{\, A(x)}|}|\nabla v{\, A(x)}-\nabla w{\, A(x)}|\ |{A^T(x)}\nabla h|\ dx \nonumber \\
            &\le& \int_{B^+_R} g\big(|\nabla v{\, A(x)}| + |\nabla w{\, A(x)}| \big) \ |\nabla h|\ dx \nonumber\\
            &\le& \varepsilon \int_{B^+_R} \widetilde{G}\Big(g\big(|\nabla v{\, A(x)}|+|\nabla w{\, A(x)}|\big)\Big)\ dx  + c(\varepsilon) \int_{B^+_R} G(|\nabla h|)\ dx \nonumber \\
            &\le& \varepsilon \int_{B^+_R} G\big(|\nabla v{\, A(x)}|+|\nabla w{\, A(x)}|\big)\ dx + c(\varepsilon) \int_{B^+_R} G(|\nabla h|)\ dx
        \end{eqnarray}
    Separately, the property $(G_1)$ combined with Hölder inequality ensure 
        \begin{eqnarray}\label{Gh}
            \int_{B^+_R} G(|\nabla h|)\ dx &=& \int_{B^+_R\cap\{|\nabla h|\le 1\}} G(|\nabla h|)\ dx + \int_{B^+_R\cap\{|\nabla h|> 1\}} G(|\nabla h|)\ dx \nonumber\\
            &\le& G(1)(g_0+1) \left(\int_{B^+_R} |\nabla h|^{\delta+1}\ dx + \int_{B^+_R} |\nabla h|^{g_0+1}\ dx\right) \nonumber\\
            &\le& G(1)(g_0+1) \left[\left(\int_{B^+_R} |\nabla h|^{q}\ dx\right)^{\frac{\delta+1}{q}} R^{n\left(1-\frac{\delta+1}{q}\right)} + \left(\int_{B^+_R} |\nabla h|^{q}\ dx\right)^{\frac{g_0+1}{q}} R^{n\left(1-\frac{g_0+1}{q}\right)}\right] \nonumber\\
            &\le& G(1)(g_0+1)\xi_1\left(\|\nabla h\|_{L^q(B_R^+)}\right) R^{n\left(1-\frac{g_0+1}{q}\right)}.
        \end{eqnarray}
    By combining \eqref{213}--\eqref{Gh} we obtain
        \begin{eqnarray}\label{216}
            I_A&\le& \mathrm{C}\left(\varepsilon \int_{B^+_R} G\big(|\nabla v{\, A(x)}|+|\nabla w{\, A(x)}|\big)\ dx  + c(\varepsilon)\xi_1\left(\|\nabla h\|_{L^q(B_R^+)}\right) R^{n\left(1-\frac{g_0+1}{q}\right)}\right),
        \end{eqnarray}
    for some constant $\mathrm{C}=\mathrm{C}(n,\delta,g_0,G(1))>0$. Furthermore, by the minimality of $w$, and similarly to \eqref{Gh}, we have
        \begin{eqnarray*}
            \int_{B^+_R} G(|\nabla w{\, A(x)}|)\ dx &\le& \int_{B^+_R} G(|\nabla v{\, A(x)}- \nabla h{\, A(x)}|)\ dx\\
            &\le& \mathrm{C} \left( \int_{B^+_R} G\big(|\nabla v{\, A(x)}|\big)\ dx + \xi_1\left(\|\nabla h\|_{L^q(B_R^+)}\right) R^{n\left(1-\frac{g_0+1}{q}\right)}\right),
        \end{eqnarray*}
    which, combined with \eqref{216}, ensures 
        \begin{eqnarray}\label{Gv-w2}
            I_A &\le& \mathrm{C}\left(\varepsilon \int_{B^+_R} G\big(|\nabla v{\, A(x)}|\big)\ dx + c(\varepsilon)\xi_1\left(\|\nabla h\|_{L^q(B_R^+)}\right) R^{n\left(1-\frac{g_0+1}{q}\right)}\right).
        \end{eqnarray}
    Finally, we use $(g_3)$, \eqref{Gw}, Lemma \ref{monotonicity} and \eqref{Gv-w2} to get
        \begin{eqnarray*}
            \int_{B^+_\rho} G(|\nabla v{\, A(x)}|)\ dx &\le& \int_{B^+_\rho} |\mathbf{V}_G(\nabla v{\, A(x)})|^2\ dx \\
            &\le& \int_{B^+_R} |\mathbf{V}_G(\nabla v{\, A(x)}) - \mathbf{V}_G(\nabla w{\, A(x)})|^2\ dx + (g_0+1)\int_{B^+_\rho} G(|\nabla w{\, A(x)}|)\ dx \\
            &\le& I_A + C \left(\dfrac{\rho}{R}\right)^n \int_{B^+_R} G(|\nabla w|)\ dx \\
            &\le& \mathrm{C}\left(\left[\left(\dfrac{\rho}{R}\right)^n + \varepsilon\right] \int_{B^+_R} G\big(|\nabla v{\, A(x)}|\big)\ dx +\xi_1\left(\|\nabla h\|_{L^q(B_R^+)}\right) R^{n\left(1-\frac{g_0+1}{q}\right)}\right),
        \end{eqnarray*}
    where $\mathrm{C}=\mathrm{C}(\delta,g_0,n,\|A\|_{C^0},G(1))>0$. Therefore, the desired inequality follows from the iteration Lemma \ref{interationlemma}.
\end{proof}

\begin{proposition}[{\bf Harnack inequality for $\mathfrak{a}$-harmonic functions}]\label{a-harmonicHarnack}
Let $u\ge 0$ be a weak $\mathfrak{a}$-harmonic function in $B_R^+$, namely
\[
-\mathrm{div}\ \mathfrak{a}(x,\nabla u)=0 \quad \text{in } B_R^+,
\]
where the vector field $\mathfrak{a}$ is defined by \eqref{defvecfield-a}.  
Assume that the function $g:[0,\infty)\to[0,\infty)$ satisfies the standard structural conditions:
\begin{itemize}
\item $g\in C^1((0,\infty))$, $g(0)=0$, and $g$ is strictly increasing;
\item there exist constants $1<p\le q<\infty$ and $0<\nu\le L$ such that
\[
\nu t^{p-1}\le g(t)\le L(1+t^{q-1}), \qquad
\nu t^{p-2}\le g'(t)\le L(1+t^{q-2}), \quad t>0.
\]
\end{itemize}
Then, for every $B_{2\rho}(x_0)\Subset B_R^+$, there exists a constant $\mathrm{C}_{\mathrm{H}}>0$,
depending only on $n,p,q,\nu,L$, and the ellipticity constants of $A$, such that
\[
\sup_{B_\rho(x_0)} u (x)\le \mathrm{C}_{\mathrm{H}} \inf_{B_\rho(x_0)} u(x).
\]
\end{proposition}

\begin{proof}
Under the standing assumptions on the matrix field $A$ and on the function $g$, the vector field $\mathfrak{a}(x,\eta)$ satisfies uniform ellipticity and Orlicz-type growth conditions. More precisely, there exist positive constants $c_1,c_2>0$ depending only on the ellipticity bounds of $A$, such that for all $(x,\eta)\in B_R^+\times\mathbb{R}^n$ one has
    $$\mathfrak{a}(x,\eta)\cdot\eta \ge c_1\, g(|\eta|)|\eta|\quad \mbox{and} \quad |\mathfrak{a}(x,\eta)|\le c_2\, g(|\eta|).$$
In addition, for each fixed $x\in B_R^+$, the mapping $\eta \mapsto \mathfrak{a}(x,\eta)$ is monotone, in the sense that
\[
\big(\mathfrak{a}(x,\eta)-\mathfrak{a}(x,\xi)\big)\cdot(\eta-\xi)\ge 0
\quad \text{for all } \eta,\xi\in\mathbb{R}^n.
\]
These structural properties place the operator associated with 
$\mathfrak{a}$ within the class of uniformly elliptic operators with nonstandard (Orlicz-type) growth and variable coefficients. Consequently, the proof proceeds by applying a De Giorgi–Moser iteration scheme suitably adapted to this framework, following well-established arguments for quasilinear operators with general growth conditions.
%
%
%
%
%
%
\end{proof}

The following result is a slight refinement, in our case, of the well-known Hopf’s Lemma proved by Braga and Moreira in \cite{BM}. More precisely, its proof follows directly from \cite[Theorem 3.2]{BM}, in combination with the Harnack inequality established in Proposition \ref{a-harmonicHarnack}.

\begin{lemma}[{\bf Hopf’s lemma for $\mathfrak{a}$-harmonic functions}]\label{Hopf}
    Assume that $G\in\mathcal{G}(\delta,g_0)$, and let $h$ be an $\mathfrak{a}$-harmonic function in $B^+_r(x_0)$, in the sense of \eqref{a-harmonic}, with nonnegative boundary value on $\partial B^+_r(x_0)$. Thus, it holds
        $$\mathrm{dist}\big(x_0+rx,\partial B^+_r(x_0)\big)\left(\dfrac{1}{r} \sup_{B^+_{\frac{r}{2}}(x_0)}\ h\right) \le \mathrm{C}(n,\delta,g_0)\ h(x_0+rx),$$
    where $\mathrm{dist}(\cdot,\cdot)$ denotes the distance function.
\end{lemma}
\begin{proof}
    Let $h$ be an $\mathfrak{a}$-harmonic function in $B_1$. By \cite[Theorem 3.2]{BM} with zero right-hand side, together with the Harnack inequality, see Proposition \ref{a-harmonicHarnack}, there exist universal constants $\mathrm{C}>0$ and $\varepsilon_0>0$ such that 
        \begin{eqnarray*}
            h(x) &\ge& \mathrm{C}\left(\intav{B^+_{\frac{1}{2}}} h(y)^{\varepsilon_0}\ dy\right)^{\frac{1}{\varepsilon_0}} \ \mathrm{dist} \big(x,\partial B^+_1\big)\\
            &\ge& \mathrm{C} \inf_{y \in B^+_{\frac{1}{2}}} h(y) \ \mathrm{dist} \big(x,\partial B^+_1\big) \\
            &\ge& \mathrm{C} \sup_{y \in B^+_{\frac{1}{2}}} h(y) \ \mathrm{dist} \big(x,\partial B^+_1\big).
        \end{eqnarray*}
    Finally, to obtain the corresponding estimate for any $h$ defined in $B^+_r(x_0)$, it suffices to consider the rescaled function $\widetilde{h}(x)=h(x_0+rx)$.
\end{proof}

The following lemma is inspired by \cite[see Lemma 2.5]{FSS}, where the corresponding results were established in the $p$-Laplacian setting. For the sake of completeness and to reinforce the reader’s understanding, we highlight here the key differences arising in our more general framework. 

\begin{lemma}\label{L27}
    Assume that $G\in\mathcal{G}(\delta,g_0)$. and let $w \in W^{1,G}(B_r(x_0))$ be any nonnegative function. Then, there exists $\mathrm{C}=\mathrm{C}(n,\delta,g_0)>0$ such that
        \begin{equation}\label{Hopf-estimate}
        \xi_0\left(\dfrac{1}{r}\ \sup_{B^+_{\frac{r}{2}}(x_0)}h\right)|B^+_r(x_0)\cap \{w=0\}| \le \mathrm{C}\int_{B^+_r(x_0)} G\big(|\nabla \big(w- h\big)(x){\color{blue}A(x)}|\big)\ dx,
        \end{equation}
    where $h$ is $\mathfrak{a}$-harmonic in $B^+_r(x_0)$ and is equal to $w$ on $\partial B^+_r(x_0)$, in the trace sense.
\end{lemma}
\begin{proof}
    Let $\tau \in (0,1)$ be fixed. For each $z\in \partial B^+_r(x_0)$, define $M_z \coloneqq \{t \in [\tau r, r] \; : \; w(x_0+tz)=0\}$, and set
        $$t_z \coloneqq \left\{
        \begin{array}{rc}
            \inf M_z,&\text{if}\; M_z\neq \varnothing; \\
            r,& \text{otherwise.}
        \end{array} \right.$$
    Since both $w$ and $(w-h)$ belong to $W^{1,G}(B^+_r(x_0))$, the maps $t \mapsto w(x_0+tz)$ and $t \mapsto (w-h)(x_0+tz)$ are absolutely continuous, for almost everywhere $z \in \partial B^+_r(x_0)$ (see, for instance, \cite[Theorem 4.20]{EG}). By continuity, it follows that $w(x_0+t_zz)=0$. Moreover, since $(w-h)$ is $\mathcal{H}^{n-1}$-a.e. zero on $\partial B^+_r(x_0)$, its trace being zero there, we also have $(w-h)(x_0+t_z z)=0$. Consequently, for almost every $z\in \partial B^+_r(x_0)$, and for each $t_z<r$ we obtain
        \begin{equation}\label{h0}
            h(x_0+t_zz)=\int_{t_z}^r z \cdot \nabla (w-h)(x_0+sz)\ ds.
        \end{equation}
    Applying Hölder’s inequality together with $(G_1)$, $(G_4)$, and the property of $A$, we find
        \begin{eqnarray}\label{nwh}
            \int_{t_z}^r |\nabla (w-h)(x_0+sz)|\ ds &\le& 2^{-1/2}\int_{t_z}^r |\nabla (w-h)(x_0+sz)\ A(x_0+sz)|\ ds\nonumber \\
            &\le&\mathrm{C} (r-t_z)^{\frac{g_0}{g_0+1}}\xi_0^{-1}\left(\int_{t_z}^r G\Big(|\nabla (w-h)(x_0+sz)\ A(x_0+sz)|\Big)\ ds\right),
        \end{eqnarray}
    where $\mathrm{C}=\mathrm{C}\left(\delta,g_0,\widetilde{G}(1)\right)>0$. On the other hand, using Hopf's Lemma \ref{Hopf}, we get
        \begin{equation}\label{h1}
            h(x_0+t_zz)\ge \mathrm{C}\mathrm{dist}\big(x_0+rx,\partial B^+_r(x_0)\big)\left(\dfrac{1}{r} \sup_{B^+_{\frac{r}{2}}(x_0)}\ h\right) = \mathrm{C} (r-t_z)\left(\dfrac{1}{r} \sup_{B^+_{\frac{r}{2}}(x_0)}\ h\right).
        \end{equation}
    By combining \eqref{h0}--\eqref{h1}, we deduce that
        $$(r-t_z)^{\frac{1}{g_0+1}}\left(\dfrac{1}{r} \sup_{B^+_{\frac{r}{2}}(x_0)}\ h\right) \le \mathrm{C}\ \xi_0^{-1}\left(\int_{t_z}^r G\Big(|\nabla (w-h)(x_0+sz)\ A(x_0+sz)|\Big)\ ds\right),$$
    and consequently 
        $$(r-t_z)\ \xi_0\left(\dfrac{1}{r} \sup_{B^+_{\frac{r}{2}}(x_0)}\ h\right) \le \mathrm{C}\int_{t_z}^r G\Big(|\nabla (w-h)(x_0+sz)\ A(x_0+sz)|\Big)\ ds.$$
    The last inequality is trivially satisfied if $t_z=r$. Therefore, by the same reason as in \cite[see Lemma 2.5]{FSS} we get
        \begin{eqnarray*}
            \mathrm{C}\int_{B^+_r(x_0)} G\Big(|\nabla (w-h)(x)\ A(x)|\Big) \ dx &\ge& \mathrm{C} \int_{\partial B^+_1}\int_{t_z}^r G\Big(|\nabla (w-h)(x_0+sz)\ A(x_0+sz)|\Big)\ ds \ d\mathcal{H}^{n-1}(z) \\
            &\ge& \xi_0\left(\dfrac{1}{r}\sup_{B^+_{\frac{r}{2}}(x_0)}h\right) \int_{\partial B^+_1}(r-t_z)\ d\mathcal{H}^{n-1}(z) \\
            &\ge& \xi_0\left(\dfrac{1}{r}\sup_{B^+_{\frac{r}{2}}(x_0)}h\right) \int_{B^+_r(x_0)\setminus B^+_{\tau r}(x_0)} \chi_{\{w=0\}}\ dx,
        \end{eqnarray*}
    and the desired conclusion follows by letting $\tau \to 0$.
\end{proof}

\section{Regularity near the boundary}

This section is devoted to recalling fundamental properties of $(\kappa,\beta)$-almost-minimizers and to establishing several auxiliary lemmas that will play a crucial role in the proof of our main results.

As a preliminary step, we recall some regularity results for almost-minimizers of the functional $\mathcal{J}_G$. These results were originally established in \cite{PSY} in the purely interior setting. To begin with, we introduce here and throughout the sequel the following notation. Considering the set
    $$P_{\bf u} \coloneqq \bigcup_{i=1}^m \{u_i>0\},$$
we understand the free boundary of an $(\kappa,\beta)$-almost-minimizer ${\bf u}$ by
    $$\mathfrak{F}(\mathbf{u}) \coloneqq \partial P_{\bf u} \cap \Omega.$$
Further, we will denote
    $$\mathfrak{B}_r(z)\coloneqq B_r(z)\cap\Omega, \quad \mathfrak{D}_r(z)\coloneqq B_r(z)\cap\partial\Omega, \quad  \mbox{and} \quad \omega_r(z,\mathbf{u}) \coloneqq\dfrac{|\mathfrak{B}_r(z)\cap \{|\mathbf{u}|=0\}|}{|\mathfrak{B}_r(z)|},$$
for any $z \in \mathbb{R}^n$, where $\{|\mathbf{u}|=0\}$ denotes the union of $\{u_i=0\}_{i=1}^m$. 

In order to apply the results established in Section \ref{Sec2.1}, namely Lemmas \ref{g-harmonic_control}, \ref{Hopf}, and \ref{L27}, we first describe the local flattening procedure of the boundary. To this end, let $x_0\in\partial\Omega$ be arbitrary and assume that $\Omega$ is of class $C^{1,\alpha}$ in a neighborhood of $x_0$. After a suitable affine transformation and rotation, we may assume without loss of generality that $x_0=0$ and that the inward unit normal vector to $\partial \Omega$ at $x_0$ satisfies $\nu_{\partial\Omega}(x_0)=e_n$. Under these assumptions, there exists a function $\mathfrak{f}:\mathbb{R}^{n-1}\to \mathbb{R}$ of class $C^{1,\alpha}$ such that $\mathfrak{f}(\mathbf{0})=0$, $\nabla\mathfrak{f}(\mathbf{0})=0$, and for some radius $R>0$ 
    $$\mathfrak{B}_R = \{x\in B_R \;\; : \;\; x_n>\mathfrak{f}(x')\},$$
where $x'=(x_1,\dots,x_{n-1})$.By choosing $R >0$ sufficiently small, we may further assume that
    $$|\nabla \mathfrak{f}(x')|< \dfrac{1}{2} \quad \mbox{for any}\ x' \in \mathbb{R}^{n-1} \ \mbox{with}\ |x'|<\sqrt{2}r.$$
Further, since $\partial\Omega$ is compact, we can choose a common $R>0$ that is suitable for all the boundary points.

We finally define the mapping $\Psi,\Psi^{-1}: \mathbb{R}^n \to \mathbb{R}^n$ by
    \begin{equation}\label{defPsi}
        \Psi(x)=\big(x',x_n-\mathfrak{f}(x')\big) \quad \mbox{and} \quad \Psi^{-1}(y)=\big(y',y_n+\mathfrak{f}(x')\big).
    \end{equation}
It is immediate to verify that $\Psi^{-1}$ is indeed the inverse of $\Psi$. Moreover, $\Psi$ locally flattens the boundary in the sense that
    $$\Psi(\mathfrak{D}_R) = D_R^+,$$
where $D_R^+$ denotes the flat boundary of the positive half-ball $B_R^+$. Furthermore, a direct computation shows that the following properties hold:
    \begin{itemize}
        \item $\mathrm{det} D\Psi = 1 = \mathrm{det} D\Psi^{-1}$;

        \item $\sqrt{\frac{1}{2}}|w| \le |D\Psi(x)w| \le \sqrt{\frac{3}{2}}|w|$ for any $w\in\mathbb{R}^n$ and $x \in B_{\sqrt{2}R}$;

        \item $B^+_{\frac{\rho}{\sqrt{2}}}\subset \Psi(\mathfrak{B}_\rho) \subset B^+_{\sqrt{2}\rho}$, for any $\rho \le \sqrt{2}R$.
    \end{itemize}

This construction allows us to establish the following equivalence between boundary value problems before and after flattening the boundary.

\begin{lemma}[{\bf Local flattening of the boundary}]  
    Assume that $G \in \mathcal{G}(\delta,g_0)$ and define the vector field $\mathfrak{a}:B_R^+\times\mathbb{R}^n \to \mathbb{R}^n$ as in \eqref{defvecfield-a} where, in this case, $A(x) \coloneqq D\Psi\big(\Psi^{-1}(x)\big)$, with $\Psi$ and $\Psi^{-1}$ defined in \eqref{defPsi}. If $u$ is a weak solution of
        $$\left\{ \begin{array}{rccl}
            -\mbox{div}\left(\dfrac{g(|\nabla u|)}{|\nabla u|}\nabla u\right) &=&0,& \mbox{in}\ \mathfrak{B}_R; \\
            u&=&\phi,& \mbox{in}\ \mathfrak{D}_R,
        \end{array}\right.$$
    for a given $\phi \in C^{1,\alpha}(\mathfrak{B}_R)$, then the function $\tilde{u}(y)=(u \circ\Psi^{-1})(y)$ is a weak solution of
        $$\left\{ \begin{array}{rccl}
            -\mbox{div}_y \ \mathfrak{a}(y,\nabla \tilde{u}) &=&0,& \mbox{in}\ B^+_R; \\
            \tilde{u}&=&\tilde{\phi},& \mbox{in}\ D^+_R,
        \end{array}\right.$$
    where $\tilde{\phi} \equiv \phi\circ\Psi^{-1} \in C^{1,\alpha}(B_R^+)$.

\end{lemma}
\begin{proof}
    Indeed, consider any $\varphi \in C_0^\infty(\mathfrak{B}_{\frac{R}{\sqrt{2}}})$ and define $\tilde{\varphi}= \varphi \circ \Psi^{-1} \in C_0^{1,\alpha}(B_R^+)$. Since $u$ is a weak solution, the definition of $\tilde{u}$, change variables, and the properties of $\Psi^{-1}$ yield
        \begin{eqnarray*}
            0&=& \int_{\mathfrak{B}_R} \dfrac{g(|\nabla u(x)|)}{|\nabla u(x)|}\nabla u(x) \nabla \varphi(x) \ dx\\
            &=& \int_{\mathfrak{B}_R} \dfrac{g(|\nabla \tilde{u}\big(\Psi(x)\big)|)}{|\nabla \tilde{u}\big(\Psi(x)\big)|}\nabla \tilde{u}\big(\Psi(x)\big) \nabla \tilde{\varphi}\big(\Psi(x)\big) \ dx \\ 
            &=& \int_{B^+_R} \dfrac{g(|\nabla \tilde{u}(y) \cdot D\Psi(\Psi^{-1}(y)\big)|)}{|\nabla \tilde{u}(y) \cdot D\Psi(\Psi^{-1}(y)\big)|}\nabla \tilde{u}(y) \cdot D\Psi(\Psi^{-1}(y)\big) \nabla \tilde{\varphi}(y) \cdot D\Psi^T(\Psi^{-1}(y)\big) \ dy\\
            &=& \int_{B^+_R} \dfrac{g(|\nabla \tilde{u}(y) \cdot A(y)|)}{|\nabla \tilde{u}(y) \cdot A(y)|}\nabla \tilde{u}(y) \cdot A(y) A^T(y) \nabla \tilde{\varphi}(y) \ dy.
        \end{eqnarray*}
    To obtain the previous identity for arbitrary test functions $\tilde{\varphi} \in C_0^\infty(B_R^+)$, it suffices to reverse the above change of variables argument, namely, to test the original problem with $\tilde{\varphi}\circ\Psi$. This yields the desired conclusion and completes the proof.
\end{proof}

\begin{remark}\label{remarkflattening}
    Let $\Psi$ be the flattening diffeomorphism as previously stated in \eqref{defPsi}. Then, by the change of variables $y=\Psi(x)$ and the chain rule, one has
        $$\int_{\mathfrak{B}_r(x_0)} G(|\nabla u(x)|)\ dx = \int_{B_r^+(x_0)}G(|\nabla \tilde{u}(y)\ A(y)|)\ dy.$$
    Consequently, all the results stated in Lemmas \ref{g-harmonic_control}, \ref{Hopf}, and \ref{L27} remain valid for any $g$-harmonic function $u$, after flattening the boundary, for the integral
        $$\int_{\mathfrak{B}_r(x_0)} G(|\nabla u(x)|)\ dx.$$
\end{remark}

The next two results were originally established only in the purely interior setting in \cite[Theorems 1.1 and 1.2]{PSY}. With the aid of Lemma \ref{g-harmonic_control}, we are now able to extend these results up to the boundary. Since the core of the argument follows the same strategy as in \cite{PSY}, we highlight here only the key differences; for the full proof, we refer the reader to \cite{PSY}. We begin with the boundary Hölder regularity of the almost-minimizer. From now on, given any scalar function $v$, we shall refer to $v^*$ as the $g$-harmonic replacement of $v$ in some open $U$, meaning the unique $g$-harmonic function in $U$ with the same trace as $v$ on $\partial U$.


\begin{proposition}[{\bf Hölder regularity of almost-minimizers}]\label{interiorreg}
    Assume that $G\in \mathcal{G}(\delta,g_0)$. Consider ${\bf u}=(u_1,\dots, u_m)$ a $(\kappa,\beta)$-almost-minimizer of $\mathcal{J}_G$ in $\Omega$, with some positive constant $\kappa \le \kappa_0$, exponent $\beta>0$, and the prescribed Lipschitz boundary value $\mathbf{\Phi}=(\phi_1,\dots,\phi_m)\in C^{0,1}(\Omega,\mathbb{R}^m)$. Then, ${\bf u}$ exhibits local $C^{0,\gamma}$-regularity in $\overline{\Omega}$, for any $0<\gamma<\frac{\delta}{g_0}$. More precisely, for any $x_0\in\overline{\Omega}$ and $r>0$, there exists a constant $\mathrm{C}=\mathrm{C}(n,m,\beta,\kappa_0,\delta,g_0)>0$ such that
        $$
        \| {\bf u}\|_{C^{0,\gamma}\big(\mathfrak{B}_r(x_0);\mathbb{R}^m\big)} \le \mathrm{C}G^{-1}\left( \sum_{i=1}^m \xi_1\Big(|\nabla u_i|_{L_G(\Omega)}\Big)+\lambda + \xi_1(\|\nabla \mathbf{\Phi}\|_\infty)\right),
        $$
    where $\xi_1$ is defined in $(G_1)$.
\end{proposition}
\begin{proof}
    Fix any $x_0 \in \overline{\Omega}$ and $r>0$ such that $B_r(x_0) \cap \partial \Omega \neq \varnothing$, since otherwise the result follows directly from \cite[Theorem 1.1]{PSY}. Let $u_i^*$ be the $g$-harmonic replacement of $u_i$ in $\mathfrak{B}_r(x_0)$. In particular, $u_i^* = \phi_i$ on the ``flat'' portion of the boundary $\mathfrak{D}_r(x_0)=\partial \mathfrak{B}_r(x_0)\cap\partial\Omega$. In this setting, Lemma \ref{g-harmonic_control} and Remark \ref{remarkflattening} guarantee that, for any $\rho \in (0,r]$,
        $$\sum_{i=1}^m\int_{\mathfrak{B}_\rho(x_0)}  G(|\nabla u_i|) \ dx \le \mathrm{C} \left(\dfrac{\rho}{r}\right)^{\tilde{n}}\sum_{i=1}^m\int_{\mathfrak{B}_r(x_0)}G(|\nabla u_i|)dx \\+ \mathrm{C}(\lambda+1)\xi_1\left(\|\nabla \phi_i\|_q\right)\ \rho^{\tilde{n}},$$
    where $\tilde{n}=n\left(1-\frac{g_0+1}{q}\right)$ and $q>g_0+1$ is such that $\nabla \mathbf{\Phi} \in L^q(\mathfrak{B}_r(x_0);\mathbb{R}^m)$, which is possible since $\mathbf{\Phi} \in C^{0,1}$.

    Next, set $\tau \in (0,1)$ and $\mathbf{V}_G$ as the excess function defined in \eqref{defVg}. Thus, by the previous inequality, properties of $\mathbf{V}_G$ and minimality of $u_i^*$ (see the proof of Theorem 1.1 in \cite{PSY} for details), we have
        \begin{eqnarray*}
            \sum_{i=1}^m \int_{\mathfrak{B}_{\tau r}(x_0)} G(|\nabla u_i|)\ dx &\le& \mathrm{C}\kappa r^\beta \sum_{i=1}^m \int_{\mathfrak{B}_{r}(x_0)} G(|\nabla u_i|)\ dx + \mathrm{C}\lambda r^{\tilde{n}} + \mathrm{C}\tau^{\tilde{n}} \sum_{i=1}^m \int_{\mathfrak{B}_{r}(x_0)} G(|\nabla u_i|)\ dx \\
            && + \xi_1\left(\|\nabla \mathbf{\Phi}\|_q\right)(r\tau)^{\tilde{n}},
        \end{eqnarray*}
    which produces
        \begin{eqnarray}
            \sum_{i=1}^m \int_{\mathfrak{B}_{\tau r}(x_0)} G(|\nabla u_i|)\ dx &\le& \tau^{\tilde{n}+\gamma -1}\Big( r^\beta \kappa_0 \mathrm{C}_* \tau^{1-\gamma-\tilde{n}} + \mathrm{C}_*\tau^{1-\gamma}\Big) \sum_{i=1}^m \int_{\mathfrak{B}_{r}(x_0)} G(|\nabla u_i|)\ dx \nonumber \\
            && + \xi_1\left(\|\nabla \mathbf{\Phi}\|_q\right)\Big(\lambda r^{\tilde{n}}+(r\tau)^{\tilde{n}}\Big), \nonumber
        \end{eqnarray}
    for any $\gamma \in (0,1)$ and some $\mathrm{C}_*>1$. Next, if we get $\tau \in (0,1)$ such that $\mathrm{C}_* \tau^{1-\gamma} \le 1/2$ and $0<r\le R_0\le1$ such that $R_0^\beta\mathrm{C}_*\kappa_0\tau^{1-\gamma+ 2\tilde{n}} \le 1/2$, the previous inequality becomes
        $$\sum_{i=1}^m \int_{\mathfrak{B}_{\tau r}(x_0)} G(|\nabla u_i|)\ dx \le\tau^{\tilde{n}+\gamma -1} \sum_{i=1}^m \int_{\mathfrak{B}_{r}(x_0)} G(|\nabla u_i|)\ dx  + \xi_1\left(\|\nabla \mathbf{\Phi}\|_q\right)\Big(\lambda r^{\tilde{n}}+(r\tau)^{\tilde{n}}\Big).$$
    Note that $R_0$ does depend on $q$, but in such a way that $0<R_0\mathrm{C}_*\kappa_0\tau^{1-\gamma-2n} \le 1/2$ when $q$ goes to $\infty$.
    
    Proceeding inductively, for any $k \in \mathbb{N}$
        \begin{eqnarray}\label{taukestimate}
            \sum_{i=1}^m \int_{\mathfrak{B}_{\tau^k r}(x_0)} G(|\nabla u_i|)\ dx &\le& \tau^{k(\tilde{n}+\gamma -1)} \sum_{i=1}^m \int_{\mathfrak{B}_{r}(x_0)} G(|\nabla u_i|)\ dx \nonumber \\
            && + \xi_1\left(\|\nabla \mathbf{\Phi}\|_q\right)r^{\tilde{n}}\left(\lambda\dfrac{\tau^{k(\tilde{n}+\gamma-1)}-\tau^{k\tilde{n}}}{\tau^{\tilde{n}+\gamma-1}-\tau^{\tilde{n}}} +\dfrac{\tau^{k(2\tilde{n}+\gamma-1)}-\tau^{k2\tilde{n}}}{\tau^{2\tilde{n}+\gamma-1}-\tau^{2\tilde{n}}}\right). 
        \end{eqnarray}
    Now, by the choices of $\tau$ and $R_0$ we get
        $$\tau^{\tilde{n}+\gamma-1}-\tau^{\tilde{n}} \ge \mathrm{C}(R_0,\beta,\gamma)\left(\mathrm{C}_*-\left(\dfrac{1}{\mathrm{C}_*}\right)^{\frac{\tilde{n}}{1-\gamma}}\right)\eqcolon \widetilde{C}_1,$$
    and
        $$\tau^{2\tilde{n}+\gamma-1}-\tau^{2\tilde{n}} \ge \mathrm{C}(R_0,\beta,\gamma)\left(\mathrm{C}_*-\left(\dfrac{1}{\mathrm{C}_*}\right)^{\frac{2\tilde{n}}{1-\gamma}}\right)\eqcolon \widetilde{C}_2.$$
    Note again that both $\widetilde{C}_1$ and $\widetilde{C}_2$ depend on $\tilde{n}$, and then depend on $q$. Nevertheless, when $q \to \infty$ they converge to some positive and finite constant.

    Hence, we may estimate \eqref{taukestimate} by
        $$\sum_{i=1}^m \int_{\mathfrak{B}_{\tau^k r}(x_0)} G(|\nabla u_i|)\ dx \le \tau^{k(\tilde{n}+\gamma -1)} \sum_{i=1}^m \int_{\mathfrak{B}_{r}(x_0)} G(|\nabla u_i|)\ dx + \dfrac{\xi_1\left(\|\nabla \mathbf{\Phi}\|_q\right)}{\widetilde{C}_1 + \widetilde{C}_2} \Big(\lambda (\tau^kr)^{\tilde{n}+\alpha-1} + (\tau^kr)^{2\tilde{n}+\alpha-1}\Big).$$
    Now, consider a natural number $k$ such that $\tau^{k+1}r\le s \le \tau^kr$. Then, the previous inequality and the same reasoning as in \cite{PSY} lead us to
        \begin{equation}\label{topassthelimit}
            \sum_{i=1}^m \int_{\mathfrak{B}_{s}(x_0)} G(|\nabla u_i|)\ dx \le \dfrac{\mathrm{C}}{R_0^{\tilde{n}+\gamma-1}}\left(\dfrac{s}{r}\right)^{\tilde{n}+\gamma-1}\sum_{i=1}^m \int_{\mathfrak{B}_{r}(x_0)} G(|\nabla u_i|)\ dx + \dfrac{\mathrm{C}}{\widetilde{C}_1+\widetilde{C}_2} \xi_1\left(\|\nabla \mathbf{\Phi}\|_q\right) (\lambda+1) s^{\tilde{n}+\gamma-1},
        \end{equation}
    for any $0<s<r\le1$. Passing to the limit of $q \to \infty$ in \eqref{topassthelimit} we obtain
        $$\sum_{i=1}^m \int_{\mathfrak{B}_{s}(x_0)} G(|\nabla u_i|)\ dx \le \mathrm{C}\left(\dfrac{s}{r}\right)^{n+\gamma-1}\sum_{i=1}^m \int_{\mathfrak{B}_{r}(x_0)} G(|\nabla u_i|)\ dx + \mathrm{C} \xi_1\left(\|\nabla \mathbf{\Phi}\|_\infty\right) (\lambda+1) s^{n+\gamma-1}.$$
    As a consequence, applying a standard covering argument, we obtain for any $\gamma \in (0,1)$
        \begin{equation}\label{Remarkone}
            \sum_{i=1}^m \intav{\mathfrak{B}_r(x_0)}G(|\nabla u_i|)\ dx \le \mathrm{C} \left( \sum_{i=1}^m \int_{\Omega} G(|\nabla u_i|)\ dx + \lambda +\xi_1\left(\|\nabla \mathbf{\Phi}\|_\infty\right) \right)r^{\gamma-1},
        \end{equation}
    where $\mathrm{C}>0$ depends on $n,m,\delta,g_0,\beta$ and $\kappa_0$. The conclusion follows by choosing $\gamma \in \left(0,\frac{\delta}{g_0}\right)$ and invoking \cite[Lemma 2.1]{PSY}.
\end{proof}

Next, we turn to a H\"{o}lder estimate for the gradient, valid away from the free boundary. As in the previous case, this type of result was proved in \cite[see Theorem 1.2]{PSY} only in the purely interior setting and under stronger assumptions on the ellipticity constants $\delta$ and $g_0$. Here, we present a more general and less restrictive formulation, which remains valid up to the boundary, as follows.

\begin{theorem}[{\bf $C^{1, \alpha}$ of almost-minimizers}]\label{gradHold}
     Suppose $G\in \mathcal{G}(\delta,g_0)$, and let ${\bf u} = (u_1, \dots, u_m)$ be a $(\kappa,\beta)$-almost-minimizer of $\mathcal{J}_G$ in $\Omega$, with some positive constant $\kappa \le \kappa_0$, exponent $\beta > 4\left(1-\frac{\delta}{g_0}\right)>0$, and the prescribed Lipschitz boundary value $\mathbf{\Phi}=(\phi_1,\dots,\phi_m)\in C^{0,1}(\Omega,\mathbb{R}^m)$. Then, ${\bf u}$ exhibits local $C^{1,\tilde{\mu}}$-regularity in points of $\overline{\Omega}$ which does not belong to $\overline{\mathfrak{F}(\mathbf{u})}$. More precisely, for any $x_0 \in \overline{\Omega} \setminus \overline{\mathfrak{F}(\mathbf{u})}$ and $\mathfrak{B}_r(x_0) \subset P_{\mathbf{u}}$, there exists an exponent $\tilde{\mu} = \tilde{\mu}(\delta, g_0, n, \beta) > 0$ and a constant $\mathrm{C} = \mathrm{C}(\delta, g_0, n, m, \kappa_0, \beta) > 0$ such that
        $$\|{\bf u}\|_{C^{1,\tilde{\mu}}(\mathfrak{B}_r(x_0);\mathbb{R}^m)} \le \mathrm{C}\left( \sum_{i=1}^m\int_\Omega G\big(\vert\nabla u_i\vert\big) \ dx + \lambda + \xi_1(\|\nabla \mathbf{\Phi}\|_\infty)\right).$$
\end{theorem}
\begin{proof}
    First, consider any $x_0 \in \overline{\Omega}\setminus \overline{\mathfrak{F}(u)}$ and $B_r(x_0)\cap\Omega \subset P_{\mathbf{u}}$, where $0<r<1$ without loss of generality. Let us assume that $B_r(x_0) \cap \partial \Omega \neq \varnothing$. Otherwise, the result follows directly from \cite[Theorem 1.2]{PSY}. 
    
    For any $i\in \{1,\dots,m\}$, let $u_i^*$ be the $g$-harmonic replacement of $u_i$ in $\mathfrak{B}_r(x_0)$. In particular, $u_i^* = \phi_i$ on the ``flat'' boundary $\mathfrak{D}_r(x_0)=\partial \mathfrak{B}_r(x_0)\cap\partial\Omega$. In this case, Lemma \ref{g-harmonic_control} and Remark \ref{remarkflattening} ensure
        $$\intav{\mathfrak{B}_s(x_0)} G(|\nabla u_i^*|)\ dx \le \mathrm{C}  \left(\left(\dfrac{s}{r}\right)^n\intav{\mathfrak{B}_r(x_0)}G(|\nabla u_i^*|)\ dx + s^n \xi_1(\|\nabla \phi_i\|_\infty) \right),$$
    for any $s \in (0,r]$. Hence, the very definition of $\mathbf{V}_G$, the minimality of $u_i^*$ in the $G$-energy sense, and \eqref{Remarkone} yield
        \begin{eqnarray}\label{VG32}
            \intav{\mathfrak{B}_s(x_0)} \Big|\mathbf{V}_G(\nabla u_i^*) - \big(\mathbf{V}_G(\nabla u_i^*)\big)_{\mathfrak{B}_s(x_0)}\Big|^2 \ dx &\le& \mathrm{C}\intav{\mathfrak{B}_s(x_0)} G(|\nabla u_i^*|)\ dx \nonumber \\ 
            &\le& \mathrm{C}  \left(\left(\dfrac{s}{r}\right)^n\intav{\mathfrak{B}_r(x_0)}G(|\nabla u_i^*|)\ dx + s^n \xi_1(\|\nabla {\phi_i}\|_\infty) \right) \nonumber \\
            &\le& \mathrm{C} \left[\left(\dfrac{s}{r}\right)^n\left(\int_\Omega G(|\nabla u_i|)\ dx +\lambda\right)r^{\alpha-1} + s^n \xi_1(\|\nabla {\phi_i}\|_\infty)\right] \nonumber \\
            &\le& \mathrm{C}\left(\dfrac{s}{r}\right)^n \left(\int_\Omega G(|\nabla u_i|)\ dx + \lambda + \xi_1(\|\nabla \phi_i\|_\infty)\right)r^{\alpha-1},
        \end{eqnarray}
    for any $0<s<r<1$. Furthermore, by a reasoning similar to that in \cite[Theorem 1.2]{PSY}, we have
        \begin{eqnarray}\label{VG33}
            \intav{\mathfrak{B}_{s}(x_0)} \big|\big(\mathbf{V}_G(\nabla u_i^*)\big)_{\mathfrak{B}_{s}(x_0)} - \big(\mathbf{V}_G(\nabla u_i)\big)_{\mathfrak{B}_{s}(x_0)}\big|^2 dx &\le& \mathrm{C} \intav{\mathfrak{B}_{s}(x_0)}\big|\mathbf{V}_G(\nabla u_i^*) - \mathbf{V}_G(\nabla u_i)\big|^2 dx \nonumber \\
            &\le& \mathrm{C}\kappa r^\beta \left( \int_\Omega G \big(|\nabla u_i|\big)\ dx + \lambda\right) r^{\alpha-1}.
        \end{eqnarray}  
    Therefore, by combining \eqref{VG32} and \eqref{VG33}, we obtain
        \begin{eqnarray}
            \sum_{i=1}^m\intav{\mathfrak{B}_{s}(x_0)}\big|\mathbf{V}_G(\nabla u_i) - \big(\mathbf{V}_G(\nabla u_i)\big)_{\mathfrak{B}_{s}(x_0)}\big|^2 dx &\le& \sum_{i=1}^m\intav{\mathfrak{B}_{s}(x_0)}\big|\mathbf{V}_G(\nabla u_i) - \mathbf{V}_G(\nabla u_i^*)\big|^2 dx \nonumber \\
            &&+ \sum_{i=1}^m\intav{\mathfrak{B}_{s}(x_0)}\big|\mathbf{V}_G(\nabla u_i^*) - \big(\mathbf{V}_G(\nabla u_i^*)\big)_{\mathfrak{B}_{s}(x_0)}\big|^2 dx \nonumber \\
            &&+ \sum_{i=1}^m\intav{\mathfrak{B}_{s}(x_0)}\big|\big(\mathbf{V}_G(\nabla u_i^*)\big)_{\mathfrak{B}_{s}(x_0)} - \big(\mathbf{V}_G(\nabla u_i)\big)_{\mathfrak{B}_{s}(x_0)}\big|^2 dx \nonumber \\
            &\le& \mathrm{C} \left(\sum_{i=1}^m \int_\Omega G \big(|\nabla u_i|\big)\ dx + \lambda + \xi_1(\|\nabla \mathbf{\Phi}\|_\infty)\right) \nonumber\\
            && \hspace{3.5cm}\cdot\left( \left(\dfrac{r}{s}\right)^n r^{\beta+\alpha-1} + \left(\dfrac{s}{r}\right)^n r^{\alpha-1}\right), \nonumber
        \end{eqnarray}
    for any $0<s<r<1$. Finally, choosing $s=r^{1+\theta}$, with $\theta \coloneqq \frac{\beta}{2n}$ we observe that
        \begin{eqnarray*}
            \left(\dfrac{r}{s}\right)^n r^{\beta+\alpha-1} + \left(\dfrac{s}{r}\right)^n r^{\alpha-1} &=& s^{\frac{-\theta n +\beta+\alpha-1}{1+\theta}} + s^{\frac{\theta n +\alpha-1}{1+\theta}}  \\
            &=& 2s^{\frac{\frac{\beta}{2}+\alpha-1}{1+\theta}}.
        \end{eqnarray*}
    Since by hypothesis $\beta>4\left(1-\frac{\delta}{g_0}\right)$, we may choose $\alpha$ close enough to $\frac{\delta}{g_0}$ such that $\beta > 4(1-\alpha)$, which ensures 
        $$\dfrac{1-\alpha}{1+\theta} < \dfrac{\frac{\beta}{2}+\alpha-1}{1+\theta},$$
    hence
        $$\sum_{i=1}^m\intav{\mathfrak{B}_{s}(x_0)}\big|\mathbf{V}_G(\nabla u_i) - \big(\mathbf{V}_G(\nabla u_i)\big)_{\mathfrak{B}_{s}(x_0)}\big|^2 dx \le \mathrm{C} \left(\sum_{i=1}^m \int_\Omega G \big(|\nabla u_i|\big)\ dx + \lambda + \xi_1(\|\nabla \mathbf{\Phi}\|_\infty)\right) s^{\mu},$$
    where $\mu = \frac{1-\alpha}{1+\theta} \in (0,1)$. Next, the Morrey and Campanato space embedding theorem (see, for example, \cite[Chapter 5]{G}) implies that $\mathbf{V}_G(\nabla u_i)$ is $C^{0,\mu}$-regular in $\mathfrak{B}_r(x_0)$. Since $\mathbf{V}_G^{-1}$ is H\"{o}lder continuous (see, for instance, \cite[Lemma 2.10]{DSV}), this implies that $\nabla \mathbf{u}$ is $C^{0,\tilde{\mu}}$-regular in $\mathfrak{B}_r(x_0)$, for some $\tilde{\mu}\in(0,1)$, and the result follows.
\end{proof}

As a direct consequence of the previous result, we have the following:

\begin{corollary}\label{MainCorollary}
    Suppose $G\in \mathcal{G}(\delta,g_0)$, and let ${\bf u} = (u_1, \dots, u_m)$ be a $(\kappa,\beta)$-almost-minimizer of $\mathcal{J}_G$ in $\Omega$, with some positive constant $\kappa \le \kappa_0$, $\beta > 4\left(1-\frac{\delta}{g_0}\right)>0$, and the prescribed Lipschitz boundary value $\mathbf{\Phi}=(\phi_1,\dots,\phi_m)\in C^{0,1}(\Omega,\mathbb{R}^m)$. Assume that $\mathfrak{B}_1({\bf 0}) \subset P_{\mathbf{u}}$. Then,
        $$
        |\nabla {\bf u}({\bf 0})|\le \mathrm{C}\cdot \Big( \xi_1\big(\|{\bf \nabla u}\|_{L^G(B_1({\bf0}),\mathbb{R}^m)}\big) + \xi_1(\|\nabla \mathbf{\Phi}\|_\infty) + \lambda\Big),
        $$
    for a universal constant $\mathrm{C}=\mathrm{C}(n,m,\kappa_0,\beta,\delta,g_0)>0$.
\end{corollary}

Next, we establish three auxiliary technical results that will play a crucial role in deriving the linear growth of almost-minimizers.

\begin{lemma}\label{Linffctrl}
    Suppose that $\Omega$ is a $C^1$-domain in $\mathbb{R}^n$. Furthermore, consider $G \in \mathcal{G}(\delta,g_0)$, with $\delta>1$, $\mathbf{u} = (u_1,\dots,u_m)$ a $(\kappa,\beta)$-almost-minimizer of $\mathcal{J}_G$ in $\Omega$, with some positive constant $\kappa \le \kappa_0$ and exponent $\beta>0$, and the prescribed Lipschitz boundary value $\mathbf{\Phi}\in C^{0,1}(\Omega,\mathbb{R}^m)$. Define $u^*_i$ to be the $g$-harmonic replacement of $u_i$ in $\mathfrak{B}_r(z)$, $z \in \overline{\mathfrak{F}(\mathbf{u})}$, and $\mathbf{v}=(u^*_1,\dots,u^*_m)$. Then,
    \begin{equation}\label{Bound-Lemma3.1}
        \displaystyle\|\mathbf{u}-\mathbf{v}\|_{L^\infty\left(\mathfrak{B}_{\frac{r}{2}}(z),\mathbb{R}^m\right)} \le \mathrm{C} \left(r^{\frac{\delta\left(\beta+n+\frac{\delta}{g_0}+g_0\right)}{g_0n+\delta(g_0+1)}}+r^{\frac{\delta\left(n+g_0+1\right)}{g_0n+\delta(g_0+1)}}\omega_r(z,\mathbf{u})^{\frac{\delta}{g_0n+\delta(g_0+1)}}\right),
        \end{equation}
    for any $\gamma \in \left(0,\frac{\delta}{g_0}\right)$, and $\mathrm{C}=\mathrm{C}(n,m,\delta,g_0,\kappa_0,\beta,\|\nabla \mathbf{\Phi}\|_\infty,\lambda,\|\nabla \mathbf{u}\|_{L^G(\Omega)})>0$.
\end{lemma}
\begin{proof}
    Indeed, since $u_i=u_i^*$ on $\partial \mathfrak{B}_r(z)$, in the trace sense, we can use $\mathbf{v}$ in the definition of $(\kappa,\beta)$-almost-minimizer of $\mathbf{u}$, and the $G$-energy minimality of $u_i^*$ to get
        \begin{eqnarray*}
            \sum_{i=1}^m \int_{\mathfrak{B}_r(z)} \Big(G(|\nabla u_i|) - G(|\nabla u_i^*|)\Big)\ dx &\le& \kappa r^\beta \sum_{i=1}^m\int_{\mathfrak{B}_r(z)} G(|\nabla u_i^*|)\ dx + \kappa r^\beta \lambda \int_{\mathfrak{B}_r(z)} \chi_{\{|\mathbf{v}|>0\}}\ dx \\
            && \hspace{4cm}+ \lambda \int_{\mathfrak{B}_r(z)} \Big(\chi_{\{|\mathbf{v}|>0\}} - \chi_{\{|\mathbf{u}|>0\}}\Big)\ dx \\
            &\le& \kappa r^\beta \sum_{i=1}^m\int_{\mathfrak{B}_r(z)} G(|\nabla u_i|)\ dx + \kappa r^\beta \lambda |\mathfrak{B}_r(z)| + \lambda |\mathfrak{B}_r(z)\cap \{|\mathbf{u}|=0\}| \\
            &=& \kappa r^\beta |\mathfrak{B}_r(z)| \left(\sum_{i=1}^m\intav{\mathfrak{B}_r(z)} G(|\nabla u_i|)\ dx +\lambda\right) + \lambda \omega_r(z,\mathbf{u})|\mathfrak{B}_r(z)|. 
        \end{eqnarray*}
    Furthermore, by \eqref{Remarkone} we have
        \begin{eqnarray*}
            \sum_{i=1}^m \int_{\mathfrak{B}_r(z)} \Big(G(|\nabla u_i|) - G(|\nabla u_i^*|)\Big)\ dx &\le& \mathrm{C}\kappa r^{\beta+n} \left(\left[\sum_{i=1}^m\xi_1\Big(|\nabla u_i|_{L_G(\Omega)}\Big)+\lambda \right]r^{\gamma-1} + \lambda r^{\gamma-1}\right) \\
            && \hspace{4cm} + \lambda \omega_r(z,\mathbf{u})|\mathfrak{B}_r(z)| \\
            &\le& \mathfrak{L} r^{\beta +n +\gamma -1} + \mathrm{C}(n)\lambda r^n\omega_r(z,\mathbf{u}),
        \end{eqnarray*}
    for any $\gamma \in (0,1)$, $\xi_1$ is defined in $(G_1)$, and 
        $$\mathfrak{L} = \mathrm{C}(n,\gamma,\delta,g_0,\|\nabla \mathbf{\Phi}\|_{\infty})\left[\sum_{i=1}^m\xi_1\Big(|\nabla u_i|_{L_G(\Omega)}\Big)+\lambda \right].$$

Thus, the combination of \eqref{SSMI} and \eqref{SSMII} yields
    \begin{equation}\label{Decay-Avg-gradient}
    \displaystyle\sum_{i=1}^m \int_{\mathfrak{B}_r(z)} G(|\nabla u_i - \nabla u_i^*|)\ dx \le \mathrm{C} \Big(\mathfrak{L} r^{\beta +n +\gamma -1} + \lambda r^n\omega_r(z,\mathbf{u})\Big).
    \end{equation}
Since $u_i-u_i^* \in W_0^{1,G}(\mathfrak{B}_r(z))$, \eqref{Poincare} and $(G_1)$ ensure that
    \begin{equation}\label{esti1}
        \sum_{i=1}^m \int_{\mathfrak{B}_r(z)} G(|u_i - u_i^*|)\ dx \le \mathrm{C} r^{g_0+1}\Big( r^{\beta +n +\gamma -1} + r^n\omega_r(z,\mathbf{u})\Big)
    \end{equation}
where $\mathrm{C}=\mathrm{C}(\mathfrak{L},\lambda)>0$.

Next, fix a point $y_0 \in \mathfrak{B}_{\frac{r}{2}}(z)$ and set $\mu_i \coloneqq |(u_i-u_i^*)(y_0)|$, for any $i \in \{1,\dots,m\}$. Since both $u_i$ and $u_i^*$ are uniformly Hölder continuous, we can exploit this regularity to estimate their difference. More precisely, for any exponent $\gamma \in \Big(0,\tfrac{\delta}{g_0}\Big)$, the Hölder continuity yields
    $$|(u_i-u_i^*)(x)| \ge |(u_i-u_i^*)(y_0)| - h_0|x-y_0|^\gamma, \quad \mbox{for any}\ x \in B_{r_i}(y_0),$$
where $h_0\coloneqq \max\{h_i\}_{i=1}^m>0$ and $h_i$ is a uniform bound for the Hölder seminorm of $u_i - u_i^*$, for each $i=1,\dots,m$. Choosing the radius $r_i \coloneqq \min \left\{\frac{r}{4}, \left(\frac{\mu_i}{2h_0}\right)^{1/\gamma}\right\}$, we see that for every $x \in B_{r_i}(y_0)$, the inequality refines to
    \begin{equation}\label{mui}
        |(u_i-u_i^*)(x)| \ge \dfrac{\mu_i}{2}.
    \end{equation}
Moreover, by construction we also have $B_{r_i}(y_0) \subset \mathfrak{B}_r(z)$. Thus, by combining \eqref{mui}, \eqref{esti1}, and the monotonicity of $G$ we get
    $$r_i^n G(\mu_i) \le \mathrm{C}\Big( r^{\beta +n +\gamma +g_0} + r^{n+g_0+1}\omega_r(z,\mathbf{u})\Big).$$
Now, suppose that $r_i = \left(\frac{\mu_i}{2h_0}\right)^{1/\gamma}$. In this case, 
    $$\mu_i^{\frac{n}{\gamma}}G(\mu_i) \le \mathrm{C}h_0^{\frac{n}{\gamma}}\Big( r^{\beta +n +\gamma +g_0} + r^{n+g_0+1}\omega_r(z,\mathbf{u})\Big).$$
If $\mu_i \in [0,1]$, $(G_1)$ and \eqref{ine} yield
    \begin{eqnarray*}
        \mu_i^{\frac{n}{\gamma}}G(\mu_i)    &\ge& \mu_i^{\frac{n}{\gamma}}\dfrac{\xi_0\left(\mu_i\right)G(1)}{g_0+1} \\
        &=& \dfrac{\mu_i^{\frac{n}{\gamma}+g_0+1}}{g_0+1}G(1),
    \end{eqnarray*}
hence
    $$\mu_i \le \mathrm{C} \Big(r^{\frac{\gamma(\beta+n+\gamma+g_0)}{n+\gamma g_0+\gamma}} + r^{\frac{\gamma(n+g_0+1)}{n+\gamma g_0+\gamma}}\omega_r(z,\mathbf{u})^{\frac{\gamma}{n+\gamma g_0+\gamma}}\Big).$$
If $\mu_i>1$, by \eqref{ine} we get $\mu_i^{\frac{n}{\gamma}+1}G(1) \le \mu_i^{\frac{n}{\gamma}}G(\mu_i)$. Thus,
    $$\mu_i \le \mathrm{C}\Big( r^{\frac{(\beta+n+\gamma+g_0)\gamma}{n+\gamma}} + r^{\frac{(n+g_0+1)\gamma}{n+\gamma}}\omega_r(z,\mathbf{u})^{\frac{\gamma}{n+\gamma}}\Big).$$

Next suppose that $r_i = \frac{r}{4}$. In this situation, we have
    $$G(\mu_i) \le \mathrm{C}\Big(r^{\beta+\gamma+g_0} + r^{g_0+1}\omega_r(z,\mathbf{u})\Big).$$
If $\mu_i \in [0,1]$, by using again $(G_1)$ and \eqref{ine} ensure
    $$\mu_i \le \mathrm{C} \Big(r^{\frac{\beta+\gamma+g_0}{g_0+1}} + r\omega_r(z,\mathbf{u})^{\frac{1}{g_0+1}}\Big).$$
Otherwise, \eqref{ine} assures that
    $$\mu_i \le \mathrm{C} \Big(r^{\beta+\gamma+g_0} + r^{g_0+1}\omega_r(z,\mathbf{u})\Big).$$

In any case, since $\omega_r(z,\mathbf{u}) \le 1$ and we may consider $r\le 1$, we get
    $$\mu_i \le \mathrm{C} \Big(r^{\frac{\gamma(\beta+n+\gamma+g_0)}{n+\gamma( g_0+1)}} + r^{\frac{\gamma(n+g_0+1)}{n+\gamma( g_0+1)}}\omega_r(z,\mathbf{u})^{\frac{\gamma}{n+\gamma( g_0+1)}}\Big).$$
Finally, choosing $\gamma>0$ close enough to $\frac{\delta}{g_0}$ the result follows.
\end{proof}

\begin{proposition}\label{LGctrl}
    Assume that $\Omega$ is a $C^1$-domain in $\mathbb{R}^n$, $G \in \mathcal{G}(\delta,g_0)$, with $\delta>1$, and let $\mathbf{u} = (u_1,\dots,u_m)$ be a $(\kappa,\beta)$-almost-minimizer of $\mathcal{J}_G$ in $\Omega$, with some positive constant $\kappa \le \kappa_0$ and exponent $\beta>0$. Define $u^*_i$ to be the $g$-harmonic replacement of $u_i$ in $\mathfrak{B}_r(z)$, $z \in \overline{\mathfrak{F}(\mathbf{u})}$. Then,
        $$\sum_{i=1}^m\int_{\mathfrak{B}_r(z)} G\big(|(u_i-u_i^*)^+|\big)\ dx \le \mathrm{C}\kappa r^{\beta+\delta+1} \left(r^n + \sum_{i=1}^m \int_{\mathfrak{B}_r(z)}G\big(|\nabla u_i|\big)\ dx\right),$$
    where $\mathrm{C}=\mathrm{C}(\delta,g_0,n)>0$, and $f^+$ denotes the positive part of the function $f$.
\end{proposition}
\begin{proof}
    We begin by introducing the following auxiliary functions, defined componentwise for each $i=1,\dots,m$:
    $$w_i(x) \coloneqq \left\{
        \begin{array}{rcl}
            u_i(x) &\text{for} & x \in \Omega \setminus \mathfrak{B}_r(z); \\
            \min\{ u_i(x),u_i^*(x)\} & \text{for} & x \in \mathfrak{B}_r(z).
        \end{array} \right.$$
    By construction $w_i \in W^{1,G}(\mathfrak{B}_r(z))$ and $w_i\equiv u_i$ on $\partial \mathfrak{B}_r(z)$. We further introduce the vector notation $\mathbf{w} \coloneqq (w_1,\dots,w_m)$, and define for each component the set 
        $$W_i \coloneqq \{x \in \mathfrak{B}_r(z) \; : \; w_i(x) \neq u_i(x)\} = \{x \in \mathfrak{B}_r(z) \; : \; u_i^*(x) \leq u_i(x)\}.$$
    Notice that the gradients satisfy the relation 
        $$\nabla w_i =  \left\{
        \begin{array}{rcl}
            \nabla u_i &\text{a.e. in} & \mathfrak{B}_r(z)\setminus W_i; \\
            \nabla u_i^* & \text{a.e. in} & W_i.
        \end{array} \right.$$    
    Moreover, since $\chi_{\{|\mathbf{w}|>0\}} \le \chi_{\{|\mathbf{u}|>0\}}$, we get
        \begin{eqnarray}
            \mathcal{J}_G\big(\mathbf{w}, \mathfrak{B}_r(z)\big) - \mathcal{J}_G\big(\mathbf{u}, \mathfrak{B}_r(z)\big) &\le& \sum_{i=1}^m\int_{\mathfrak{B}_r(z)}\Big( G(|\nabla w_i|) - G(|\nabla u_i|) \Big)\ dx\nonumber \\
            &=&\sum_{i=1}^m\int_{W_i}\Big( G(|\nabla u_i^*|) - G(|\nabla u_i|) \Big)\ dx.
        \end{eqnarray}
    Next, we claim that
        \begin{equation}\label{Wiine}
            \sum_{i=1}^m \int_{W_i} |\mathbf{V}_G(\nabla u_i) - \mathbf{V}_G(\nabla u_i^*)|^2\ dx \le \mathrm{C} \sum_{i=1}^m \int_{W_i} \Big(G(|\nabla u_i|) - G(|\nabla u_i^*|)\Big)\ dx,
        \end{equation}
    where $V_G$ is defined in \eqref{defVg}. To prove this inequality, let us argue as follows. For simplicity of notation, we introduce $v_i(x) \coloneqq u_i^*(x) + \big( u_i(x) -u_i^*(x) \big)^+$ for any $x \in \mathfrak{B}_r(z)$. In particular, $v_i \in W^{1,G}(\mathfrak{B}_r(z))$, and its trace on $\partial \mathfrak{B}_r(z)$ coincides both with $u_i^*$ and $u_i$. By $g$-harmonicity of $u_i^*$ we have
        $$\int_{\mathfrak{B}_r(z)} \dfrac{g(|\nabla u_i^*|)}{|\nabla u_i^*|}\nabla u_i^* \cdot \nabla \big(u_i - u_i^*\big)^+\ dx = 0.$$
    Hence, the definition of $v_i$ yields
        $$\int_{\mathfrak{B}_r(z)} \dfrac{g(|\nabla u_i^*|)}{|\nabla u_i^*|}\nabla u_i^* \cdot \nabla v_i\ dx = \int_{\mathfrak{B}_r(z)} \dfrac{g(|\nabla u_i^*|)}{|\nabla u_i^*|}\nabla u_i^* \cdot \nabla u_i^*\ dx.$$
    Equivalently, this can be written as
        $$\int_{\mathfrak{B}_r(z)} \dfrac{g(|\nabla u_i^*|)}{|\nabla u_i^*|}\nabla u_i^* \cdot \big(\nabla u_i^* - \nabla v_i\big)\ dx = 0.$$
    At this point, we apply inequality \eqref{SSMI} with the pair $(u_i^*, v_i)$, and the previous equality ensures that
        $$\sum_{i=1}^m\int_{\mathfrak{B}_r(z)} |\mathbf{V}_G(\nabla v_i) - \mathbf{V}_G(\nabla u_i^*)|^2 dx \le \mathrm{C} \sum_{i=1}^m\int_{\mathfrak{B}_r(z)} \Big( G(|\nabla v_i|) - G(|\nabla u_i^*|)\Big)\ dx.$$
    Finally, by the very definition of $v_i$, the above inequality reduces exactly to \eqref{Wiine}, which establishes the claim.

    Further, by definition of $w_i$ and the $G$-minimality of $u_i^*$ in $\mathfrak{B}_r(z)$ we have
        \begin{equation}\label{wiui}
            \int_{\mathfrak{B}_r(z)} G(|\nabla w_i|)\ dx \le \int_{\mathfrak{B}_r(z)} G(|\nabla u_i|)\ dx + \int_{\mathfrak{B}_r(z)} G(|\nabla u_i^*|)\ dx \le 2\int_{\mathfrak{B}_r(z)} G(|\nabla u_i|)\ dx.
        \end{equation}
    Thus, by combining  \eqref{SSMII}, \eqref{Wiine}, and \eqref{wiui} we get
        \begin{eqnarray*}
            \sum_{i=1}^m \int_{W_i} G(|\nabla u_i - \nabla u_i^*|)\ dx &\le& \sum_{i=1}^m \int_{W_i} |\mathbf{V}_G(\nabla u_i) - \mathbf{V}_G (\nabla u_i^*)|^2\ dx \\
            &\le& \sum_{i=1}^m \int_{W_i} \Big(G(|\nabla u_i|) - G(|\nabla u_i^*|)\Big)\ dx \\
            &\le& \mathrm{C}\kappa r^\beta \mathcal{J}_G(\mathbf{w},\mathfrak{B}_r(z)) \\
            &\le& \mathrm{C}\kappa r^\beta \left(r^n + \sum_{i=1}^m \int_{\mathfrak{B}_r(z)} G(|\nabla w_i|)\ dx \right) \\ 
            &\le& \mathrm{C}\kappa r^\beta \left(r^n + \sum_{i=1}^m \int_{\mathfrak{B}_r(z)} G(|\nabla u_i|)\ dx \right).
        \end{eqnarray*}

    Finally, since $u_i \le u_i^*$ in $\mathfrak{B}_r(z)\setminus W_i$ we ensure that $\nabla \big((u_i-u_i^*)^+\big) = \chi_{W_i} (\nabla u_i - \nabla u_i^*)$ a.e. in $\mathfrak{B}_r(z)$. Thus, we may apply $(G_1)$ and the Poincaré inequality \eqref{Poincare} to the function $(u_i - u_i^*)^+$ 
        \begin{eqnarray*}
            \dfrac{\mathrm{C}}{r^{\delta+1}} \sum_{i=1}^m \int_{\mathfrak{B}_r(z)} G(|(u_i-u_i^*)^+|)\ dx &\le&  \mathrm{C} \sum_{i=1}^m \int_{\mathfrak{B}_r(z)} G(|\nabla (u_i-u_i^*)^+|)\ dx \\
            &=& \mathrm{C} \sum_{i=1}^m \int_{W_i} G(|\nabla u_i-\nabla u_i^*|)\ dx \\
            &\le& \mathrm{C}\kappa r^\beta \left(r^n + \sum_{i=1}^m \int_{\mathfrak{B}_r(z)} G(|\nabla u_i|)\ dx \right),
        \end{eqnarray*}
    and this completes the proof.
\end{proof}

\begin{corollary}\label{improveLinfctrl}
    Assume that $\Omega$ is a $C^1$-domain in $\mathbb{R}^n$, $G \in \mathcal{G}(\delta,g_0)$, with $\delta>1$, and let $\mathbf{u} = (u_1,\dots,u_m)$ be a $(\kappa,\beta)$-almost-minimizer of $\mathcal{J}_G$ in $\Omega$, with some positive constant $\kappa \le \kappa_0$ and exponent $\beta>0$. Define $u^*_i$ to be the $g$-harmonic replacement of $u_i$ in $\mathfrak{B}_r(z)$, $z \in \overline{\mathfrak{F}(\mathbf{u})}$. Then,
        \begin{equation}\label{Est-Corollary3.1}
        \|(u_i-u_i^*)^+\|_{L^\infty\left(\mathfrak{B}_{\frac{r}{2}}(x)\right)}\le \mathrm{C} r^{\frac{\delta\left(\beta+n+\frac{\delta}{g_0}+g_0\right)}{g_0n+\delta(g_0+1)}},
        \end{equation}
    where $\mathrm{C} = \mathrm{C}(n,m,\delta,g_0,\lambda)>0$.     
\end{corollary}
\begin{proof}
    As before, consider $W_i=\{ x \in \mathfrak{B}_{r}(z) \, : \, u_i^*(x) < u_i(x)\}$, $y_0 \in W_i$, and define
        $$ \mu_i\coloneqq(u_i-u_i^*)(y_0).$$
    Consider also $h_i$ the uniform bound for the Hölder seminorm of $u_i - u_i^*$, for each $i=1,\dots,m$, and $h_0 \coloneqq \max\{h_i\}_{i=1}^m$. Thus, setting $r_i \coloneqq \min \left\{\frac{r}{4}, \left(\frac{\mu_i}{2h_0}\right)^{1/\gamma}\right\}$, for any $\gamma \in (0,\frac{\delta}{g_0})$, we can proceed as in the proof of Lemma \ref{Linffctrl} to get 
        $$ (u_i-u_i^*)(x) \geq \frac{\mu_i}{2}, \qquad \text{in} \quad B_{r_i}(y_0), $$
and therefore $B_{r_i}(y_0) \subset W_i$.
Since $G$ is non-decreasing, Proposition \ref{LGctrl} and \eqref{Remarkone} ensure
    \begin{eqnarray*}
        |B_1|r_i^nG(\mu_i) &\le& \int_{B_{r_i}(y_0)} G(\mu_i)\ dx\\
        &\le& \mathrm{C} \sum_{i=1}^m\int_{\mathfrak{B}_r} G(|u_i-u_i^*|)\ dx \\
        &\le&\mathrm{C}\kappa r^{\beta+\delta+1}\left(r^n + \sum_{i=1}^m\int_{\mathfrak{B}_r} G(|\nabla u_i|)\ dx\right) \\
        &\le& \mathrm{C}\kappa r^{\beta+\delta+1} \Big(r^n+r^{n+\gamma-1}\Big) \\
        &\le& \mathrm{C} r^{\beta+n+\delta+\gamma},
    \end{eqnarray*}
for any $\gamma \in \left(0,\frac{\delta}{g_0}\right)$. By similar arguments employed in the proof of Lemma \ref{Linffctrl} we ensure that
    $$\mu_i \le \mathrm{C} r^{\frac{\delta\left(\beta+n+\frac{\delta}{g_0}+g_0\right)}{g_0n+\delta(g_0+1)}}.$$
This completes the proof.
\end{proof}

\section{Boundary Lipschitz regularity}

As stated earlier in this manuscript, our goal is to establish the Lipschitz regularity of almost-minimizers at contact points. Note that Lipschitz regularity (and even $C^{1,\tilde{\alpha}}$-regularity) near the Dirichlet boundary, in the absence of free boundary points, has already been obtained in Theorem \ref{gradHold}. To address the behavior at every point of the domain, including contact points, we divide the argument into two steps. The first step is devoted to establishing the linear growth of the almost-minimizer at contact points. The second step then uses this growth estimate to prove the Lipschitz continuity itself.

\begin{proof}[{\bf Proof of Theorem \ref{thm:boundary-Lip}}] \mbox{}

\noindent
\begin{tabular}{|l|}
	\hline
	{\bf Step one}\\
	\hline
\end{tabular} \underline{Linear growth at free boundary points}

In this step we want to prove that, for any $z \in \mathfrak{F}(\mathbf{u})$, $\mathbf{u}$ exhibits linear growth at $z$; that is, there exist constants $r_0>0$ and $\mathrm{C}>0$, depending only on $\delta$, $g_0$, $n$, $\kappa$, $\lambda$, $\|\nabla \mathbf{u}\|_{L^{G}(\Omega;\mathbb{R}^m)}$, and $\|\mathbf{\Phi}\|_{C^{0,1}(\Omega;\mathbb{R}^m)}$, such that 
    \begin{equation}\label{linear-growth}
        \sup_{\mathfrak{B}_r(z)} |\mathbf{u}| \le \mathrm{C}r, \qquad \text{for all } r \le r_0.
    \end{equation}



    Indeed, for each $i=1,\dots,m$, let $u^*_i$ be the g-harmonic replacement of $u_i$ in $\mathfrak{B}_r(z)$. Invoking \eqref{Hopf-estimate} with $h = u^*_i$ and $w = u_i$, and applying estimate \eqref{Decay-Avg-gradient}, we obtain
\begin{equation}\label{eq:41}
\xi_0\left(
\frac{1}{r}\,
\sup_{\mathfrak{B}_{\frac{r}{2}}(z)} u^*_i
\right)
\,\bigl| \mathfrak{B}_{r}(z)\cap\{u_i = 0\}\bigr|
\;\le\;
\mathrm{C} \Big(r^{\beta +n +\gamma -1} + \lambda r^n\omega_r(z,\mathbf{u})\Big),
\end{equation}
for any $0<\gamma \le \frac{\delta}{g_0}$, where $
\mathrm{C} = \mathrm{C}\bigl(\gamma, \delta,g_0,n,\kappa, \|\mathbf{\Phi}\|_{C^{0,1}(\Omega;\mathbb{R}^m)}, 
\|\nabla \mathbf{u}\|_{L^{G}(\Omega;\mathbb{R}^m)}, \lambda \bigr)>0.$

We next claim that
\begin{equation}\label{eq:42}
\frac{1}{r}\sup_{\mathfrak{B}_{\frac{r}{2}}(z)} u^*_i \;\le\; \xi^{-1}_0(2\mathrm{C}\lambda),
\qquad \text{for all } 0<r<r_0,
\end{equation}
for some sufficiently small $r_0>0$.

To establish this, argue by contradiction and assume that
\begin{equation}\label{eq:43}
\frac{1}{r}\sup_{\mathfrak{B}_{\frac{r}{2}}(z)} u_i^* \;>\; \xi^{-1}_0(2\mathrm{C}\lambda).
\end{equation}
Then \eqref{eq:41} yields
\begin{equation}\label{eq:44}
\mathrm{C}\lambda\,\omega_r(z,\mathbf{u}) \;\le\; \mathrm{C} r^{\beta+\gamma-1}.
\end{equation}
Inserting \eqref{eq:44} into \eqref{Bound-Lemma3.1}, gives

\begin{equation}\label{eq:45}
\begin{array}{ccl}
\|u_i - u^*_i\|_{L^\infty\left(\mathfrak{B}_{\frac{r}{2}}(z)\right)} &
\;\le\; &
\mathrm{C} \left(r^{\frac{\delta\left(\beta+n+\frac{\delta}{g_0}+g_0\right)}{g_0n+\delta(g_0+1)}}+r^{\frac{\delta\left(n+g_0+1\right)}{g_0n+\delta(g_0+1)}}\omega_r(z,\mathbf{u})^{\frac{\delta}{g_0n+\delta(g_0+1)}}\right)\\
&\leq & C\left(
r^{\frac{\delta\left(\beta+n+\frac{\delta}{g_0}+g_0\right)}{g_0n+\delta(g_0+1)}} 
+ 
r^{\frac{\delta\left(n+g_0+\beta+\gamma\right)}{g_0n+\delta(g_0+1)}}\right)\\
&\;\le\;&
\mathrm{C}r^{\frac{\delta\left(n+g_0+\beta+\gamma\right)}{g_0n+\delta(g_0+1)}}.
\end{array}
\end{equation}

Since $z \in \mathfrak{F}(\mathbf{u})$, combining \eqref{eq:45} with the Harnack inequality gives
\[
\sup_{\mathfrak{B}_{\frac{r}{2}}(z)} u^*_i
\;\le\;
\mathrm{C}_{\mathrm{H}}\,\inf_{\mathfrak{B}_{\frac{r}{2}}(z)} u^*_i
\;\le\;
\mathrm{C} r^{\frac{\delta\left(n+g_0+\beta+\gamma\right)}{g_0n+\delta(g_0+1)}},
\]
which, for sufficiently small $r_0$, and since $\beta > \left(\frac{g_0}{\delta}-1\right)n+1$, contradicts \eqref{eq:43}. This proves the claim \eqref{eq:42}. Consequently, $u^*_i$ enjoys linear growth at $z$, namely
\begin{equation}\label{eq:46}
0 \;\le\; u_i^*(x) \;\le\; 2\xi_0^{-1}(2\mathrm{C}\lambda)\,|x-z|,
\qquad x \in \mathfrak{B}_{\frac{r_0}{2}}(z).
\end{equation}

We now complete the proof.  
On the set $\{u_i \le u^*_i\}$, the linear growth of $u_i$ follows directly from \eqref{eq:46}.  
On the complementary set $\{u_i - u^*_i > 0\}$, using estimate \eqref{Est-Corollary3.1} with $\beta > \left(\frac{g_0}{\delta}-1\right)n+1$ gives
\begin{equation}\label{eq:47}
\|(u_i - u^*_i)^{+}\|_{L^\infty\left(\mathfrak{B}_{r}(z)\right)} 
\;\le\; 
\mathrm{C} r^{\frac{\delta\left(\beta+n+\frac{\delta}{g_0}+g_0\right)}{g_0n+\delta(g_0+1)}}
\;\le\; \mathrm{C}r,
\end{equation}
or, equivalently,
\[
0 \;\le\; u_i(x) \;\le\; \mathrm{C} r + \|u^*_i\|_{L^\infty(\mathfrak{B}_{r}(z))},
\qquad x\in\{u_i - u^*_i>0\}.
\]

Finally, invoking \eqref{eq:46} once more yields the desired linear growth of $u_i$ at $z$.

\medskip
\noindent
\begin{tabular}{|l|}
	\hline
	{\bf Step two}\\
	\hline
\end{tabular} \underline{Lipschitz regularity}


Fixed an arbitrary point $x_{0} \in \{|\mathbf{u}|>0\}$, our aim is to estimate $|\nabla \mathbf{u}(x_{0})|$. We proceed as follows. Let $r_{0}>0$ be the constant given by the linear growth stated in step one (see \eqref{linear-growth}), and define
\[
\mathrm{d} := \operatorname{dist}(x_{0}, \mathfrak{F}(\mathbf{u})).
\]
We distinguish two cases.

\medskip
\noindent\textbf{Case I: \underline{If $\mathrm{d} \le \tfrac{1}{2} r_{0}$}.}  
Choose $y_{0} \in \partial B_{\mathrm{d}}(x_{0}) \cap \mathfrak{F}(\mathbf{u})$.  
By step one, we have
\[
|\mathbf{u}(x)| \le \mathrm{C} |x - y_{0}| \le \mathrm{C} \mathrm{d},
\qquad \text{for all } x \in \mathfrak{B}_{\mathrm{d}}(x_{0}).
\]
Define the rescaled function
\[
\mathbf{u}_{\mathrm{d}}(x) \coloneqq \frac{\mathbf{u}(x_{0}+ \mathrm{d} x)}{\mathrm{d}},
\]
which satisfies $|\mathbf{u}_{\mathrm{d}}| \le C$ and is an almost-minimizer of the functional
\[
\mathbf{v} \mapsto \mathcal{J}_G({\bf v};\Omega) \coloneqq \int_\Omega  \left(\sum_{i=1}^mG\big(|\nabla v_i(x)|\big) + \lambda \chi_{\{|{\bf v}|>0\}}(x)\right) dx ,
\]
in $B_{1} \cap \frac{1}{\mathrm{d}}(\Omega - x_{0})$, with constant $\kappa \mathrm{d}^{\beta}$ and exponent $\beta$.  

Therefore, we may apply Corollary~\ref{MainCorollary} to conclude that
\[
|\nabla \mathbf{u}(x_{0})|
= |\nabla \mathbf{u}_{\mathrm{d}}(\mathbf{0})|
\le \widetilde{\mathrm{C}},
\]
where $\widetilde{\mathrm{C}}$ depends only on 
$\delta$, $g_0$, $m$, $n$, $\kappa_{0} r_{0}^{\beta}$, $\beta$, $\lambda$, 
$\|\nabla \mathbf{u}\|_{L^{G}(\Omega;\mathbb{R}^{m})}$, 
$\|\nabla\mathbf{\Phi}\|_{\infty}$, 
and the regularity of $\Omega$.

\medskip
\noindent\textbf{Case II: \underline{$\mathrm{d} > \tfrac{1}{2} r_{0}$}.}  
In this regime, we introduce the rescaled function
\[
\mathbf{u}_{r_0}(x) \coloneqq \frac{\mathbf{u}(x_{0}+ r_0 x)}{r_0},
\]
which is an almost-minimizer of the same functional in the domain $\mathfrak{B}_1$, with constant  
$\kappa_{r_0} \coloneqq \kappa r_0^\beta$ and the same exponent $\beta$.  
Therefore, we may apply Corollary \ref{MainCorollary} directly to obtain
\[
|\nabla \mathbf{u}(x_{0})| = |\nabla \mathbf{u}_{r_0}(\mathbf{0})| \le \widehat{\mathrm{C}},
\]
where the positive constant $\widehat{\mathrm{C}}$ depends only on 
$\delta$, $g_0$, $m$, $n$, $\kappa_{0} r_{0}^{\beta}$, $\beta$, $\lambda$, 
$\|\nabla \mathbf{u}\|_{L^{G}(\Omega;\mathbb{R}^{m})}$, 
$\|\nabla \mathbf{\Phi}\|_{\infty}$, 
and on the regularity of $\Omega$.
\end{proof}

\subsection*{Acknowledgments}

\hspace{0.4cm}  Pedro Fellype Pontes was partially supported by NSFC (W2433017), and BSH (2024-002378). J.V. da Silva has received partial support from CNPq-Brazil under Grant No. 307131/2022-0, FAEPEX-UNICAMP (Project No. 2441/23, Special Calls - PIND - Individual Projects, 03/2023), and Chamada CNPq/MCTI No. 10/2023 - Faixa B - Consolidated Research Groups under Grant No. 420014/2023-3. Minbo Yang was partially supported by the National Natural Science Foundation of China (12471114), and Zhejiang Natural Science Foundation (LZ26A010002).

\end{document}